\documentclass[11pt]{article}
\usepackage[]{graphicx}
\usepackage[]{color}
\usepackage{booktabs}
\usepackage{lscape}
\usepackage{alltt}
\usepackage{bbm}
\usepackage{color}
\usepackage[english]{babel}
\usepackage{graphics}
\usepackage{graphicx}
\usepackage{epsfig}
\usepackage{lscape}
\usepackage{amsmath,amsfonts,amssymb, amsthm}
\usepackage{eurosym}
\usepackage{natbib}
\usepackage{url}
\usepackage{hyperref}
\usepackage{setspace}
\usepackage{caption}
\usepackage{subcaption}
\usepackage{arydshln}
\usepackage{rotating}
\usepackage{todonotes}
\usepackage{pdflscape}
\usepackage{etoolbox}
\usepackage{float}
\usepackage{pgfplots}
\usepackage{pifont}
\usepackage[title]{appendix}
\usepackage{multirow}
\usepackage{algorithm}
\usepackage{algpseudocode}
\usepackage{bbm}
\usepackage{appendix}

\AtBeginEnvironment{quote}{\small}

\newtheorem{theorem}{Theorem}

\newtheorem{corollary}[theorem]{Corollary}
\newtheorem{proposition}[theorem]{Proposition}
\newtheorem{lemma}[theorem]{Lemma}

\newtheorem{exmp}[theorem]{Example}
\newtheorem{definition}[theorem]{Definition}
\newtheorem{remark}[theorem]{Remark}
\def\@begintheorem#1#2{\trivlist \item[\hskip \labelsep{\bf #1\ #2.}]\sl}
\renewenvironment{abstract}
 {\begin{center}\normalsize\textsc{Abstract}%
 \end{center}\begin{quote}\normalsize}
 {\end{quote}}

\usepackage{bm}

\renewcommand{\epsilon}{\varepsilon}

\newcommand{\xmark}{\ding{55}}%

\DeclareFontFamily{U}{mathx}{}
\DeclareFontShape{U}{mathx}{m}{n}{<-> mathx10}{}
\DeclareSymbolFont{mathx}{U}{mathx}{m}{n}
\DeclareMathAccent{\widecheck}{0}{mathx}{"71}
\usepackage{tikz}

\bibpunct{(}{)}{;}{a}{}{}

\title{{\sc\Large Bounds on the Distribution of a Sum of Two Random Variables:}\\
{\it\Large Revisiting a problem of Kolmogorov with application to\\[-8pt]
Individual Treatment Effects}} 

\date{\today}

\footnotesize

\author{Zhehao Zhang and Thomas S. Richardson\\
University of Washington, Seattle \\
Department of Statistics}

\begin{document}

\maketitle
\thispagestyle{empty}

\begin{abstract}
\thispagestyle{plain}
\noindent
 We revisit the following problem, proposed by Kolmogorov: given prescribed marginal distributions $F$ and $G$ for random variables $X,Y$ respectively, characterize the set of compatible distribution functions for the sum $Z=X+Y$. Bounds on the distribution function for $Z$ were first given by \cite{makarov1982estimates} and \cite{ruschendorf1982random} independently. \cite{frank1987best} provided a solution to the same problem using copula theory. However, though these authors obtain the same bounds, they make different assertions concerning their sharpness. In addition, their solutions leave some open problems in the case when the given marginal distribution functions are discontinuous. These issues have led to some confusion and erroneous statements in subsequent literature, which we correct.

Kolmogorov's problem is closely related to inferring possible distributions for individual treatment effects $Y_1 - Y_0$  given the marginal distributions of $Y_1$ and $Y_0$; the latter being identified from a randomized experiment. We use our new insights to sharpen and correct the results due to \cite{fan2010sharp} concerning individual treatment effects, and to fill some other logical gaps.

\end{abstract}



\section{Introduction}
The question of the best possible bounds for the distribution function of the sum of two random variables whose individual distribution functions are fixed was originally raised by A.N. Kolmogorov and was first solved by \cite{makarov1982estimates} and (independently)  \cite{ruschendorf1982random}. Using copula theory \citep{sklar1959fonctions}, \cite{frank1987best} reframed this question and provided an elegant proof that the bounds were achievable in certain settings. 

Makarov-type bounds are widely studied in the
field of optimal transport, quantitative
risk management, and banking (see \citealt{puccetti2015extremal}, Example 4.2 and \citealt{puccetti2024beautiful}, Section 3 for a detailed discussion). Generalizations of the bounds to the sum of more than $2$ random variables or vector-valued random variables can be found in \cite{embrechts2013model} and \cite{li1996bounds}. Computational aspects of the bounds have been discussed in \cite{puccetti2024beautiful}, \cite{hofert2017improved}, \cite{puccetti2012computation} and \cite{kreinovich2006computing}. 

Notwithstanding these extensions, here we will focus on the original bivariate problem concerning the relationship of two scalar-valued random variables with fixed marginals. In this vein, \cite{williamson1990probabilistic} generalized the bounds in \cite{frank1987best} to other arithmetic operations of two random variables including subtraction, multiplication and division and claimed that these bounds are sharp. 
More recently, \cite{fan2010sharp} introduced the bounds in \cite{williamson1990probabilistic} into the context of causal inference and concluded sharp bounds on the distribution function of the additive treatment effect contrast (which corresponds to a difference between two random variables). The bounds proposed in \cite{fan2010sharp} have gained widespread traction in the literature on causal inference and  econometrics, in works such as \citealt{chiba2017sharp}, \citealt{huang2017inequality}, \citealt{lu2018treatment} and \citealt{mullahy2018individual}.

In this paper we revisit the connection between the construction of \cite{frank1987best} and Kolmogorov's original question. We distinguish between bounds that are {\em achievable} and those that although they provide an infimum or supremum -- and hence cannot be improved -- are {\em not attained} by {\em any} distribution. We characterize the circumstances under which the bounds are achievable. 

Building on \cite{makarov1982estimates} and \cite{frank1987best}, we formulate sharp bounds on the distribution function of the difference of two random variables with fixed marginals; we show how these differ from those that have appeared previously.
We further identify and address logical gaps in \cite{williamson1990probabilistic} that have propagated to some of the later literature, leading to incorrect or imprecise statements. In particular, for the distribution function of the difference of two random variables, the lower bound proposed by \cite{williamson1990probabilistic} is not sharp for measures that are not absolutely continuous with respect to the Lebesgue measure. We also identify an unnecessary exclusion in the argument given by \cite{williamson1990probabilistic}.  Finally, we apply the new bounds in the context of treatment effects and  calculate the bounds in an illustrative example.

In addition to the new results provided in the paper, we also hope to make the theorems in the existing literature more accessible. The prior results are stated in different papers that are hard to relate because there are at least four `dichotomies' that specify the nature of the bound and the probability (or feature of the distribution function) that is being bounded:
\begin{itemize}
\item[(i)] Upper bounds versus lower bounds;
\item[(ii)] Whether the `distribution function' of the sum is defined in terms of $P(Z<z)$ or $P(Z\leq z)$;
\item[(iii)] Whether stated bounds are on the left or right limits of the distribution function;
\item[(iv)] Whether a lower (upper) bound is considered to be `sharp' if it represents the infimum (supremum) of the attainable probabilities versus, in addition, requiring that the bound be {\em achievable} in that it is attained by a joint distribution with the specified margins.
\end{itemize}

These choices interact in that, for example, a bound on the left limit of `a distribution function evaluated at $z$' may correspond to a bound on $P(Z<z)$ or $P(Z\leq z)$, depending on which definition of a distribution function is used. 
Similarly, as we show in Section \ref{sec:sharpness} below, though all the bounds that we give are sharp in the sense of being supremums or infimums, only certain combinations are always achievable; in particular, there are cases where the bound on a left (right) limit is attainable, but the bound on the right (left) limit is not; see Figure \ref{fig:summary_ex11_12}. We also present new results giving general conditions on the margins under which all the bounds are achievable. 

The above distinctions can also be important in cases where, at first sight, it might not be expected. Obviously, if $Z=X+Y$ is a continuous random variable, then the left and right limits of its distribution function will agree (and hence distinctions (ii) and (iii) are moot). However, 
it is {\em not} the case, in general, that if $X$ and $Y$ are continuous random variables, then the distributions for $Z$ at which the bounds on the distribution function for $Z$ are attained will be continuous. As a consequence, even though $X$ and $Y$ may be continuous random variables, the bounds on the left and right limits of the distribution function for their sum, $Z$, may differ in terms of attainability.

As a further surprise, as we show in Theorem \ref{thm:discrete_ach}, if at least one of $X$ or $Y$ is a discrete random variable, then although the distributions for $Z$ that attain the bounds will be discontinuous,  in fact the upper and lower bounds on the left and right limits are all attainable.

In summary, we believe that the complexity within the existing literature, arising from these four dichotomies, may partly explain the origin of the suboptimal bounds stated by \cite{williamson1990probabilistic} for example. Table \ref{tab:literature} in Appendix \ref{app:table_prior} relates this paper to the notation, terminology, and results in prior works.

\medskip

The rest of the paper is organized as follows. In Section \ref{sec:sum}, we review the proof of the best possible bounds proposed in \cite{frank1987best}. In Section \ref{sec:sharpness}, we revisit the connection of \cite{frank1987best}'s result to Komogorov's question. We provide characterizations of when the bounds proposed in \cite{frank1987best} for the sum of two random variables are achievable. In Theorems \ref{thm:fns_exten_low} and \ref{thm:fns_exten_up} we summarize our new results and relate them to those of \cite{frank1987best}. Figure \ref{fig:summary_ex11_12}  displays the bounds arising from the illustrative Examples \ref{ex:one} and \ref{ex:two} considered in this section. In Section \ref{sec:frank}, we give the resulting bounds for the difference of two random variables whose individual distribution functions are fixed. We then discuss the implications for some of the results given previously by  \cite{williamson1990probabilistic}. In Section \ref{sec:fan}, we set up the problem of causal inference and revisit the bounds in \cite{fan2010sharp}. We demonstrate that with discrete random variables the previously stated bounds are not sharp.


\section{Bounds on the sum of two random variables}
\label{sec:sum}
Given marginal distribution functions $F,G$ for random variables $X,Y$ respectively, Kolmogorov's question (restated here in terms of right-continuous distribution functions) is to find functions $\underline{J}$ and $\overline{J}$ such that for all $z\in\mathbb{R}$,
\begin{align*}
    \underline{J}(z)&=\inf P(X+Y\leq z),\\
    \overline{J}(z)&=\sup P(X+Y\leq z),
\end{align*}
where the infimum and supremum are taken over all possible joint distribution functions $H(x,y)$ having the marginal cdfs $F(x)$ and $G(y)$. 

First, we review some existing results on probability distributions and copulas.

\begin{definition}\label{def:cdf}
Let $X$ be a random variable. The distribution function (or cumulative distribution function, cdf) $F$ of $X$ is defined to be $F(x)=P(X\leq x)$ for $x\in \mathbb{R}$.
\end{definition}
Note that under Definition $\ref{def:cdf}$, for any random variable $X$, $F(\cdot)$ is a right-continuous function. \cite{frank1987best} and \cite{williamson1990probabilistic} used a left-continuous version of the definition of distribution functions where they replace $P(X\leq x)$ with $P(X<x)$.\footnote{Some of the results we cited in this paper was originally defined in terms of the left-continuous functions $\widetilde{F}(x) = P(X<x)$ and $\widetilde{G}(x) = P(X<x)$ (which somewhat confusingly they also call ``distribution functions''). 
To be consistent in notation, $F,G$ in this paper always refer to the right-continuous distribution functions given by $F(x) = P(X\leq x)$ and likewise for $G$ and $Y$. We reserve  $\widetilde{F},\widetilde{G}$ when we need to talk about the functions given by $P(X<x)$ and $P(Y<y)$, equivalently the left hand limits of $F(x)$ and $G(y)$.}

\begin{definition}[\citealt{embrechts2013note}]\label{def:g_inv}
Let $X$ be a random variable with distribution function $F$. The generalized inverse (also known as the quantile function) $F^{-1}:[0,1]\to \overline{\mathbb{R}}=[-\infty, \infty]$ of $F$ is defined as:
$$F^{-1}(u)=\inf\{x\in\mathbb{R}, F(x)\geq u\}, u\in [0,1],$$
with $\inf\emptyset=\infty$.
\end{definition}
\begin{definition}
A two dimensional copula is a mapping $C$ from $[0,1]^2$ to $[0,1]$ satisfying the conditions:
\begin{enumerate}
\item $C(a,0)=C(0,a)=0$ and $C(a,1)=C(1,a)=a$, for all $a$ in $[0,1]$; 
\item $C(a_2,b_2)-C(a_1,b_2)-C(a_2,b_1)+C(a_1,b_1)\geq 0$ for all $a_1, a_2, b_1, b_2$ in $[0,1]$ such that $a_1\leq a_2, b_1\leq b_2$.
\end{enumerate}
\end{definition}
\begin{proposition}
According to the definition, any copula $C$ is nondecreasing in each argument, and, 
\begin{align*}
W(a, b) \leq C(a, b) \leq M (a, b)
\end{align*}
where
\[W(a, b) = \max (a + b- 1, 0),\quad M(a,b)=\min(a,b).\]
The bounds $W$ and $M$ are known as Fr\'{e}chet-Hoeffding copula bounds.
\end{proposition}

\begin{theorem}[Sklar 1959]
Consider a 2-dimensional cdf $H$ with marginals $F, G$. There exists a copula $C$, such that
$$
H(x,y)=C(F(x),G(y))
$$
for all $x,y$ in $[-\infty, \infty]$. If $F, G$ are both continuous, then $C$ is unique\footnote{Note that the condition here relates to the cdfs $F,G$ viewed as functions. We are not assuming that the random variables $X,Y$ are (absolutely) continuous (with respect to Lebesgue measure). The latter is a sufficient but not necessary condition for $F,G$ to be continuous.}; otherwise $C$ is uniquely determined only on $\operatorname{Ran} F\times \operatorname{Ran} G$, where $\operatorname{Ran} F,\operatorname{Ran} G$ denote respectively the range of the cdfs $F$ and $G$.
\end{theorem}

See \cite{sklar1959fonctions}, \cite{embrechts2013note}, \cite{schmidt2007coping} for more discussion of general $n$-dimensional copulas and Sklar's Theorem. \\

 Now we want to bound the cdf of the sum of two random variables using copulas.  We reprise the argument given by \cite{frank1987best} which 
 gives lower bounds on $P(X+Y<z)$ and upper bounds on $P(X+Y\leq z)$, and further establishes by construction that these are {\em achievable}.
 
 Let $H$ be a two-dimensional cumulative distribution function for random variables $X,Y$ with marginals $F,G$ respectively so that $H(x,y)=P(X\leq x,Y\leq y)$. By Sklar's theorem \citep{sklar1959fonctions}, there exists a copula $C$ such that $H(x,y)=C(F(x),G(y))$. Note that the cdf for the sum of two random variables $X,Y$ is fully characterized by their joint cdf $H$. Let $Z:=X+Y$ and $J$ be the cdf for $Z$. Then for any $z\in [-\infty,\infty]$,
\begin{align*}
J(z)=\iint_{x+y\leq z}dH(x,y).
\end{align*}
For the copula $C$ and marginal distribution functions $F, G$, let $\sigma_C(F,G)$ be the function defined by $\sigma_C(F,G)(-\infty)=0, \sigma_C(F,G)(\infty)=1$ and 
$$
\sigma_{C}(F, G)(z)=\iint_{x+y\leq z} d C(F(x), G(y)), \quad \text { for }-\infty<z<\infty.
$$
Since $H(x,y)=C(F(x),G(y))$, $\sigma_{C}(F, G)(z)=J(z)$ for all $z\in [-\infty,\infty]$.
We let
\begin{align*}
\tau_{C}(F, G)(z)&=\sup _{x+y=z} C(F(x), G(y)),\\
\rho_{C}(F, G)(z)&=\inf _{x+y=z} \widecheck{C}(F(x), G(y)),
\end{align*}
where
\[\widecheck{C}(a, b)=a + b-C(a, b).\]
\begin{theorem}[\citeauthor{frank1987best} Theorem 2.14]\label{Thm:franks}
We have the following bounds for any copula $C$ and arbitrary given distribution functions $F,G$ and $z\in[-\infty,\infty]$:
$$
\tau_{W} (F,G)(z) \leq \tau_{C}(F,G)(z) \leq \sigma_{C}(F,G)(z) \leq \rho_{C}(F,G)(z) \leq \rho_{W}(F,G)(z).\footnote{Proof of these bounds and visualizations can be found in \cite{frank1987best}, \cite{nelsen2006introduction}. Here only $\sigma_C$ can be interpreted as the probability that $Z\leq z$. Therefore, there are no explicit constructions establishing whether or not the bounds on $\tau_W, \rho_W$ can be achieved.}
$$
\end{theorem}
Note that $\tau_W, \rho_W$ are known functions that depend solely on the marginal cdfs $F$ and $G$.
We obtain the following bounds:
\begin{align*}
\tau_{W} (F,G)(z)\leq \sigma_{C}(F,G)(z)\leq \rho_{W}(F,G)(z).
\end{align*}
Here we will write the bounds $\tau_W, \rho_W$ explicitly. 
\begin{align}
    \tau_W(F,G)(z)&=\sup_{x+y=z}\max(F(x)+G(y)-1,0); \label{eq:low_valid}\\
    \rho_W(F,G)(z)&=\inf_{x+y=z}\left(F(x)+G(y)-\max(F(x)+G(y)-1,0)\right)\\
    &=\inf_{x+y=z}\left(F(x)+G(y)+\min(-F(x)-G(y)+1,0)\right)\\
    &=\inf_{x+y=z}\min(1,F(x)+G(y))\\
    &=1+\inf_{x+y=z}\min(0,F(x)+G(y)-1). \label{eq:up_valid}
\end{align}

Theorem \ref{Thm:franks} establishes the validity of the bounds in Equation (\ref{eq:low_valid}) and (\ref{eq:up_valid}). Thereafter, we call the bounds $\tau_W(F,G)(z)$ and $\rho_W(F,G)(z)$ the Makarov bounds
on $P(X+Y \leq z)$. \cite{frank1987best} shows that these bounds are equivalent to those given in \cite{makarov1982estimates}, which is the first paper to prove the bounds. Independently, \cite{ruschendorf1982random} also proves similar bounds and relates the bounds to the literature on infimal convolution \citep{rockafellar1997convex}.\footnote{In Appendix \ref{app:ruschendorf_vs_makarov}, we discuss the relationship between the bounds in \cite{ruschendorf1982random} and the Makarov bounds presented in this section.}
\section{Sharpness of the bounds}\label{sec:sharpness}
 To investigate the tightness of the bounds, we first distinguish three notions of sharpness.

For two random variables $X,Y$ with fixed marginals $F,G$, respectively, let $J(\cdot)$ be the distribution function of $X+Y$ and let 
$ J_{\ell}(\cdot )$ and $J^u(\cdot )$ be bounding functions such that
$ J_{\ell}(z)\leq J(z)\leq  J^u(z)$ for all $z\in\mathbb{R}$.
\begin{definition}[Achievability at a point]
    We say the lower bound $J_{\ell}(\cdot)$ is \textup{achievable at $z=z_0$} if there exists a joint distribution $H$ of $X,Y$ satisfying the marginals such that under $H$, $J(z_0)=J_{\ell}(z_0)$. The upper bound $J^u(\cdot)$ is achievable at $z=z_0$ if there exists a joint distribution $H$ of $X,Y$ satisfying the marginals such that under $H$, $J(z_0)=J^u(z_0)$.
\end{definition}
\begin{definition}[Pointwise Best-Possible]\label{def:best_p}
       We say the lower bound $J_{\ell}(\cdot)$ is \textup{pointwise best-possible} if for all $z_0\in \mathbb{R}$ and $\epsilon>0$, $J_{\ell}(z_0)+\epsilon$ will not be a valid lower bound for $J(z_0)$. In other words, for all $z_0\in \mathbb{R}$ and $\epsilon>0$, there exists a joint distribution $H$ of $X,Y$ satisfying the marginals such that under $H$, $J(z_0)< J_{\ell}(z_0)+\epsilon$. The upper bound $J^u(\cdot)$ is pointwise sharp if for all $z_0\in \mathbb{R}$ and $\epsilon>0$, $J^u(z_0)-\epsilon$ will not be a valid lower bound for $F(z_0)$.\footnote{Unlike \cite{firpo2019partial}, we differentiate between achievability and pointwise sharpness because a bound can be pointwise sharp but not necessarily achievable.}
\end{definition}
\begin{definition}[Uniformly Sharp]
    We say the lower bound $J_{\ell}(\cdot)$ is \textup{uniformly sharp of $H$} if there exists a single joint distribution $H$ of $X,Y$ satisfying the marginals such that under $H$, $J(z)=J_{\ell}(z)$ for all $z\in \mathbb{R}$. The upper bound $J^u(\cdot)$ is uniformly sharp if there exists a single joint distribution $H$ of $X,Y$ satisfying the marginals such that under $H$, $J(z)=J^u(z)$ for all $z\in \mathbb{R}$.
\end{definition}

Following these definitions, if a bound is uniformly sharp, then it is achievable for all $z\in \mathbb{R}$ and also pointwise sharp. If a bound is achievable for all $z\in \mathbb{R}$, then it is pointwise sharp. However, a pointwise sharp bound may not be achievable for all $z\in \mathbb{R}$.

\subsection{Prior results on achievability of bounds}
We first state a theorem given in \cite{nelsen2006introduction}.\footnote{\cite{frank1987best} uses an example of a degenerate distribution to show that $\tau_W$ and $\rho_W$ can be achieved for certain $F,G$ (where the bounds are uniformly sharp in the example). However, it is not clear that for arbitrary $F,G$, Kolmogorov's question is answered.} 
\begin{theorem}[\citeauthor{frank1987best}  Theorem 3.2, \citeauthor{nelsen2006introduction} Theorem 6.1.2]\label{thm:optimal}
Let $F$ and $G$ be two fixed distribution functions. For any $z\in(-\infty, \infty)$:\\
(i) There exists a copula $C_{t}$, dependent only on the value $t$ of $\tau_W(F, G)$ at $z-$, such that
\begin{align}
    \sigma_{C_{t}}(F, G)(z-)=\tau_W(F, G)(z-)=t,\label{eq:frank-lower}
\end{align}
where $z-$ is the left hand limit of the functions $
\sigma_{C_{t}}$ and $\tau_W$ as they approach $z$.\footnote{\cite{frank1987best} state their result in terms of $\widetilde{F}(z) = P(Z\!<\!z)$,
which is left-continuous. When translating the result, we need to be careful about the implications of this notational difference. In general, $\tau_W(\widetilde{F},\widetilde{G})(z-)\leq \tau_W(\widetilde{F},\widetilde{G})(z)= \tau_W(F,G)(z-)\leq \tau_W(F,G)(z)$ and $\rho_W(\widetilde{F},\widetilde{G})(z-)\leq \rho_W(\widetilde{F},\widetilde{G})(z)= \rho_W(F,G)(z-)\leq \rho_W(F,G)(z)$. See Appendix \ref{app:relation_left_right_cdf_bounds} where we prove these results.}

\noindent(ii) There exists a copula $C_r$, dependent only on the value $r$ of $\rho_W(F, G)(z)$, such that
\begin{align*}
   \sigma_{C_r}(F, G)(z)=\rho_W(F, G)(z)=r.
\end{align*}

\end{theorem}

Theorem \ref{thm:optimal} rephrases the result of \cite{frank1987best} in terms of the more common definition of right continuous distribution functions. The proof of Theorem \ref{thm:optimal} is sketched in  \cite{nelsen2006introduction}. \cite{embrechts2002correlation} Theorem 5 also sketches the proof for the lower bound (\ref{eq:frank-lower}).

Kolmogorov's question concerns the pointwise sharp bounds on the distribution function for the sum of two random variables. Statement (ii) of Theorem \ref{thm:optimal} along with Sklar's Theorem shows that the Makarov upper bound (\ref{eq:up_valid}) 
is achievable for all $z\in \mathbb{R}$ and thus is pointwise sharp. The achievability of the Makarov upper bounds can also be proved using a result in \cite{ruschendorf1983solution}. We provide a complete proof based on \cite{ruschendorf1983solution}'s argument in Appendix \ref{App:C}. We will next consider two examples which
show that statement (i) of Theorem \ref{thm:optimal} does not completely address Kolmogorov's question.
\begin{exmp}\label{ex:one}
    Let $X,Y$ be Bernoulli random variables with $p_1=0.5$ and $p_2=0.4$. Let $F,G$ be the distribution functions of $X,Y$ respectively. Suppose that we are interested in obtaining a lower bound for $P(X+Y\leq 1)$. Clearly, $\tau_W(F,G)(1)=0.6$ while $\tau_W(F,G)(1-)=0.1$. Statement $(i)$ of Theorem \ref{thm:optimal} does not provide lower bound $\tau_W(F,G)(1)$ on $P(X+Y\leq 1)=J(1)$.
\end{exmp}
\begin{exmp}[\citealp{nelsen2006introduction}] \label{ex:two}
    Let $X,Y$ be random variables with uniform distributions on $[0,1]$. Let $F,G$ be distribution functions of $X,Y$ respectively. Suppose that we wish to obtain a lower bound for $P(X+Y\leq 1)$. In this example, $\tau_W(F,G)(1)=\tau_W(F,G)(1-)=0$. We obtain the lower bound for $J(z)=P(X+Y\leq 1)=0$. Statement $(i)$ of Theorem \ref{thm:optimal} tells us that there exists a joint distribution such that $P(X+Y<1)=0$. In fact, the construction in the proof of Theorem \ref{thm:optimal}, given by \citeauthor{nelsen2006introduction}
    will result in $P(X+Y\leq 1)=1$ and $P(X+Y<1)=0$ when $X=1-Y$. However, this leaves open the question of whether $\tau_W(F,G)(1)=0$ is a tight bound on $P(X+Y\leq 1)$, either in the sense of being achievable, or more weakly, best possible.
\end{exmp}

 To summarize, statement $(i)$ of Theorem \ref{thm:optimal} leaves two questions unanswered: First, as shown in 
 Example \ref{ex:one}, when $F,G$ are not continuous, $\tau_W(F,G)(z-)$ can be different from $\tau_W(F,G)(z)$. Thus when $F$
 and $G$ are not continuous Theorem \ref{thm:optimal}
 $(i)$ does not provide any information regarding whether the bound $\tau_W(F,G)(z)$ is either achievable or best possible vis a vis $J(z)=P(X+Y\leq z)$. Second, even when $F,G$ are continuous, so that
 $\tau_W(F,G)(z-)= \tau_W(F,G)(z)$, 
 it is still possible that $\sigma_{C_t}(F,G)(z-)<\sigma_{C_t}(F,G)(z)$. Consequently, even in the continuous case, the existence of the copula $C_t$ given in
 statement $(i)$ of Theorem \ref{thm:optimal}
merely establishes that $\tau_W(F,G)(z-)$ is achievable as a lower bound
on $P(X+Y<z)$; it says nothing about the achievability (or otherwise) of  $\tau_W(F,G)(z)$ as a lower bound on $J(z)$. 

\subsection{When are the bounds achievable?}
 
 The next Theorem establishes that 
 that for all $F$, $G$ (possibly discontinuous)
 $\tau_W(F,G)(z)$ is best possible -- thus addressing the unanswered question of Kolmogorov.
 We will return to the issue of achievability of $\tau_W(F,G)(z)$ as a lower bound on $P(X+Y\leq z)$, 
in Theorems \ref{thm:discrete_ach} (discrete case) to \ref{thm:when_achi} (general case) below.

\begin{theorem}[No loose end to Kolmogorov's question]\label{thm:no_loose_end}
Let $F$ and $G$ be two fixed distribution functions. For any $z\in(-\infty, \infty)$, let $s=\tau_W(F, G)(z)$. For any $\epsilon>0 $, there exists a copula $C_{s,\epsilon}$ such that
$$
\sigma_{C_{s,\epsilon}}(F, G)(z)< s+\epsilon .
$$
In other words, the lower bound $\tau_W(F,G)(z)$ on $J(z)$ cannot be improved for any $F,G$, and is thus pointwise best possible in terms of $z$.
\end{theorem}
Before proving Theorem \ref{thm:no_loose_end}, we first state a useful lemma.
\begin{lemma}[\citeauthor{firpo2019partial} Theorem 2]\label{lemma: rc_c}
For fixed distribution functions $F,G$, the function
$$\tau_W(F,G)(z)=\sup_{x+y=z}\max\{F(x)+G(y)-1,0\}$$
is right continuous and non-decreasing for all $z$.
\end{lemma}
Lemma $\ref{lemma: rc_c}$ follows from the fact that $F,G$ are right continuous and nondecreasing. 
\begin{proof}

The proof of Theorem \ref{thm:no_loose_end} uses a similar construction as in \cite{makarov1982estimates}.
Now suppose that for a given $z$, $\tau_W(F,G)(z)=s$. Following Lemma $\ref{lemma: rc_c}$, for all $\epsilon>0$, there exists $\delta>0$ such that $\tau_W(F,G)(z+m)-s<\epsilon$ for all $0<m<\delta$. Pick any $0<m<\delta$ that meets the above condition that $\tau_W(F,G)(z+m)<s+\epsilon$. Then there exists a copula $C_{s,\epsilon}$ such that
\begin{align}
    \sigma_{C_{s,\epsilon}}(F,G)(z)&\leq \sigma_{C_{s,\epsilon}}(F,G)((z+m)-) \label{eq:non-de}\\
    &=\tau_W(F,G)((z+m)-)\label{eq:thm21i}\\
    &\leq\tau_W(F,G)(z+m)\label{eq:non-dec}\\
    &<s+\epsilon,\label{eq:construct}
\end{align}
where (\ref{eq:non-de}) follows from the fact that for any random variable $D$, $P(D\leq d)\leq P(D<d+h)$ for all $h>0$; (\ref{eq:thm21i}) follows from the optimality result in part $(i)$ of Theorem \ref{thm:optimal}; (\ref{eq:non-dec}) follows from the non-decreasing property of $\tau_W$;(\ref{eq:construct}) follows from the construction of $\tau_W(F,G)(z+m)$. This implies that for all $\epsilon>0$, we can construct a copula $C_{s,\epsilon}$ such that $\sigma_{C_{s,\epsilon}}(F,G)(z)<\tau_W(F,G)(z)+\epsilon$, meaning that $\tau_W(F,G)(z)+\epsilon$ will not be a valid lower bound on $P(X+Y\leq z)=J(z)$.
\end{proof}

Theorem \ref{thm:no_loose_end} along with Theorem \ref{thm:optimal} proves that Makarov's bounds are pointwise best-possible and thus directly address Komogorov's question. This closes the best-possible question left open by Theorem \ref{thm:optimal} (i).

\subsection{Special case: at least one discrete variable}

The following theorem states that the lower bound $\tau_W(F,G)(z)$ is achievable when we are dealing with one or more discrete random variables.
\begin{theorem}
\label{thm:discrete_ach}
    If at least one of $X,Y$ is a discrete random variable (a random variable which may take on only a countable number of distinct values), then 
    there exists a copula $C_t$, dependent only on the value of $t=\tau_W(F, G)(z)$, such that
$$
\sigma_{C_t}(F, G)(z)=\tau_W(F, G)(z)=t .
$$
In other words, the lower bound on $J(z)$ is always achievable when one of $X,Y$ is a discrete random variable.\footnote{Notice that this is Theorem \ref{thm:optimal} (i) with $z-$ replaced by $z$, which can be surprising because \cite{frank1987best} note that theorem \ref{thm:optimal} (i) cannot be strengthened ``even" when $F,G$ are continuous, which implicitly implies that it cannot be strengthened when $F,G$ are not continuous. Some misconception about this point can be found in related literatures. For example, \cite{kim2014identifying} stated that ``If the marginal distributions of $X$ and $Y$ are both absolutely continuous with respect to the Lebesgue measure on $\mathbb{R}$, then the Makarov upper bound and lower bound can be achieved", which is contradicted by Example \ref{ex:two}; Similarly, \cite{williamson1990probabilistic} stated that in order for  Theorem \ref{thm:optimal} to hold, it is necessary that $F,G$ are not both discontinuous at a point $x,y$ such that $x+y=z$; whereas
\cite{frank1987best}'s proof does not require this to hold.}
\end{theorem}
\begin{proof}
Before proving the Theorem \ref{thm:discrete_ach}, we will first prove a useful lemma.
\begin{lemma}
    For any $k<z$, if $t=\tau_W(F,G)(z)>\tau_W(F,G)(z-)$, there cannot be $x',y'$ with $x'+y'=k<z$ such that $F(x')+G(y')-1=t$.\label{lemma:klessz}
\end{lemma}
\begin{proof}
    Suppose that there exist $x',y'$ with $x'+y'=k<z$ and $F(x')+G(y')-1=t$, then by the definition of sup,
\begin{align}
    t &\leq \sup_{x+y=k}\max ({F(x)+G(y)-1},0) \label{eq:tauwk}\\ 
    &\leq \lim_{\epsilon>0,\epsilon\to 0}\sup_{x+y=z-\epsilon}\max ({F(x)+G(y)-1},0)\label{eq:limtau}
    \\&= \tau_W(F,G)(z-) \label{eq:tauwz-}
    \\&\leq\sup_{x+y=z} \max({F(x)+G(y)-1},0)=t,\label{eq:tauwz}
\end{align} 
where the inequalities (\ref{eq:limtau}) and (\ref{eq:tauwz}) follow from the fact that the function $\tau_W(F,G)(\cdot)$ is non-decreasing and (\ref{eq:tauwk}),(\ref{eq:tauwz-}) and (\ref{eq:tauwz}) are by definition of $\tau_W(F,G)(k),\tau_W(F,G)(z-),\tau_W(F,G)(z)$. Thus, we have $\tau_W(F,G)(z)=\tau_W(F,G)(z-)$, which contradicts with the hypothesis that $\tau_W(F,G)(z)>\tau_W(F,G)(z-)$.
\end{proof}
Now we will prove Theorem \ref{thm:discrete_ach}. Without loss of generality, we assume $X$ is a discrete random variable. For $t=\tau_W(F,G)(z)$, we construct the copula
$$
C_t(u, v)= \begin{cases}\operatorname{Max}(u+v-1, t), & (u, v) \text { in }[t, 1] \times[t, 1], \\ \operatorname{Min}(u, v), & \text { otherwise.}\end{cases}
$$

Figure \ref{fig:c_t_lower} illustrates the support of $C_t$ and the mass assigned by $C_t$.\footnote{In Appendix \ref{app:sim_copula}, we discuss properties of the joint distribution induced by $C_t$ and outline how to simulate from it.}
\begin{figure}
    \centering

\begin{tikzpicture}[x=0.75pt,y=0.75pt,yscale=-1,xscale=1]

\draw  (60.8,247) -- (302.8,247)(85,70.15) -- (85,266.65) (295.8,242) -- (302.8,247) -- (295.8,252) (80,77.15) -- (85,70.15) -- (90,77.15)  ;
\draw   (85,147) -- (185,147) -- (185,247) -- (85,247) -- cycle ;
\draw   (85,177) -- (155,177) -- (155,247) -- (85,247) -- cycle ;
\draw   (155,147) -- (185,147) -- (185,177) -- (155,177) -- cycle ;
\draw [color={rgb, 255:red, 255; green, 0; blue, 0 }  ,draw opacity=1 ]   (85,247) -- (155,177) ;
\draw [color={rgb, 255:red, 255; green, 0; blue, 0 }  ,draw opacity=1 ]   (155,147) -- (185,177) ;

\draw (180,251) node [anchor=north west][inner sep=0.75pt]   [align=left] {1};
\draw (70,140) node [anchor=north west][inner sep=0.75pt]   [align=left] {1};
\draw (150,250) node [anchor=north west][inner sep=0.75pt]   [align=left] {$t$};
\draw (72,168) node [anchor=north west][inner sep=0.75pt]   [align=left] {$t$};
\draw (86.9,249.85) node [anchor=north west][inner sep=0.75pt]   [align=left] {$0$};

\draw (100,198) node [color={rgb, 255:red, 255; green, 0; blue, 0 }, anchor=north west][inner sep=0.75pt]    {$H_{2}$};
\draw (167,150) node [color={rgb, 255:red, 255; green, 0; blue, 0 }, anchor=north west][inner sep=0.75pt]    {$H_{1}$};

\draw (192,76) node   [align=left] {\begin{minipage}[lt]{68pt}\setlength\topsep{0pt}
\end{minipage}};

\end{tikzpicture}

    \caption{Support of $C_t$ and mass assigned by $C_t$}
    \label{fig:c_t_lower}
\end{figure}

 Let $H_1:=\{(u,v)\in [0,1]^2 \,|\, u+v-1=t\}$, $H_2:=\{(u,v)\in [0,1]^2 \,|\, u=v<t\}$, $S_z:=\{(u,v)\in [0,1]^2 | F^{-1}(u)+G^{-1}(v)=z\}$, $S_{\bar{z}}:=\{(u,v)\in [0,1]^2 | F^{-1}(u)+G^{-1}(v)\leq z\}$, where $F^{-1}$, $G^{-1}$ are the generalized inverses defined in Definition \ref{def:g_inv}. Since $t$ is a lower bound for $P(X+Y\leq z)$, $C_t$ assigns mass at least $t$ to the set $S_{\bar{z}}$. In particular, $H_2\subseteq S_{\bar{z}}$. Thus, whether or not the lower bound on $P(X+Y\leq z)$ is achieved (by $C_t$) depends on whether the mass that $C_t$ assigns to the set $H_1\cap S_{\bar{z}}$ is $0$.

We claim that $\iint_{S_{\bar{z}}\cap H_1}dC_t(u,v)  =\iint_{S_z\cap H_1}dC_t(u,v)$. First suppose $\tau_W(F,G)(z)=\tau_W(F,G)(z-)=t$. \cite{frank1987best}, showed that $P(X+Y<z)=\tau_W(F,G)(z-)=t$ under $C_t$. Thus, under $C_t$, the set $\{(u,v)\in [0,1]^2 | F^{-1}(u)+G^{-1}(v)<z\}$ will contain all the mass in $H_2$ (equals $t$) but not any mass in $H_1$, so that the claim holds in this case. Now suppose $\tau_W(F,G)(z)>\tau_W(F,G)(z-)$, from Lemma \ref{lemma:klessz} we know that the set $\{(u,v)\in [0,1]^2 | F^{-1}(u)+G^{-1}(v)<z\}$ cannot contain mass in $H_1$. Therefore, $\iint_{S_{\bar{z}}\cap H_1}dC_t(u,v)  =\iint_{S_z\cap H_1}dC_t(u,v)$, establishing the claim.

Since $C_t$ assigns mass $(1-t)$ uniformly to $H_1$, if $S_z\cap H_1$ is empty or only contains countably many points, then $\iint_{S_z\cap H_1}dC_t(u,v)=0$, which is sufficient to establish the claim.

Since $X$ is discrete, $x$ can take at most countably many values with non-zero probability under $F$. For a given $z$, there are at most countably many points $(x,y)$ such that $x+y=z$ and $P(X=x)>0$. 

Observe that
\begin{align}\label{eq:rxy_def}
    \{(u,v)\in [0,1]^2 | F^{-1}(u)+G^{-1}(v)=z\}
= \cup_{(x,y):x+y=z} R_{xy},
\end{align}
where $R_{xy}\equiv\{(F(x-),F(x)]\times(G(y-),G(y)]\}$ and we define the sets of form $(a,a]$ as $\{a\}$ for any $a\in\mathbb{R}$, which will arise if $Y$ is not discrete.

We will show that for each $(x,y)$ with $x+y=z$, $R_{xy}\cap H_1$ contains at most one point. Since by definition of $C_t$, $t=\sup_{x+y=z}\max\{F(x)+G(y)-1,0\}$, we have $F(x)+G(y)-1\leq t$ for any $x+y=z$. For any $(u,v)\in R_{xy}$ with $u<F(x)$ or $v<G(y)$, it holds that $u+v-1<t$ and thus $(u,v)$ cannot be in $H_1$. Therefore $R_{xy}\cap H_1$ contains at most one point $(F(x),G(y))$.\footnote{Each rectangular region $R_{xy}$ can touch the line $u+v-1=t$ for at most one point because all points in $R_{xy}$ satisfy $F^{-1}(u)+G^{-1}(v)=z$ and if there is more than one point in the intersection then $t$ is not $\sup_{x+y=z}\max\{F(x)+G(y)-1,0\}$.} As a consequence by (\ref{eq:rxy_def}), there exist at most countably many points in $S_z\cap H_1$. Thus, $\iint_{S_z\cap H_1}dC_t(u,v)=0$ and
$$
\sigma_{C_t}(F, G)(z)=\tau_W(F, G)(z)=t .
$$
\end{proof}

\subsection{
$C_t$ not achieving the bound implies no other copula achieves the bound
}

Theorem \ref{thm:optimal} shows that the lower bound $\tau_W(F,G)(z-)$ on $J(z-)$ can be achieved. In fact, the proof of Theorem \ref{thm:optimal} (see  \citealt{frank1987best} and \citealt{nelsen2006introduction}) shows that the copula we constructed as $C_t$ with $t=\tau_W(F,G)(z-)$ in the proof of Theorem \ref{thm:discrete_ach} will achieve the lower bound $\tau_W(F,G)(z-)$.\footnote{In fact, as noted above, \cite{frank1987best} consider bounds on $\widetilde{J}$, \cite{nelsen2006introduction} translate the result to the standard definition $J$ but do not provide a full proof.} We further showed in Theorem \ref{thm:no_loose_end} that the lower bound $\tau_W(F,G)(z)$ on $J(z)$ cannot be improved. In example \ref{ex:two}, for $t=\tau_W(F,G)(z)=\tau_W(F,G)(z-)$, we see that $P(X+Y<z)=t$ under $C_t$ but $P(X+Y\leq z)> t$ under $C_t$. 

This raises a new question: if we care \textcolor{black}{not merely about sharpness, but also} about the achievability of the lower bound $\tau_W(F,G)$ on $J(z)$
-- rather than $J(z-)$ -- and if $C_t$ does not achieve the bound $\tau_W(F,G)$, can there be other copulas that can achieve the bound $\tau_W(F,G)$ for $J(z)$? Indeed, \cite{frank1987best} and \cite{nelsen2006introduction} both pointed out that there are other copulas beside $C_t$ that achieve the lower bound $\tau_W(F,G)(z-)$ for $J(z-)$.

The corollary of the next theorem implies that for continuous $F,G$ and an arbitrary $z$, in order to determine whether the lower bound $\tau_W(F,G)(z)$ on $J(z)$ can be achieved, we only need to determine whether it is achieved under $C_t$ for $t=\tau_W(F,G)(z)$. Theorem \ref{thm:lower_bound_ach_when_greater} along with Theorem \ref{thm:equi_C_t} further establishes this claim for arbitrary $F,G$. In other words, if the lower bound for $J(z)$ is {\em not} achieved under $C_t$, then there is {\em no} joint distribution that will achieve this lower bound.

\begin{theorem}\label{thm:equi_C_t}
        Given arbitrary $z$ and $F,G$, if $\tau_W(F,G)(z-)=\tau_W(F,G)(z)=t$ and the copula $C_t$ does not achieve the lower bound $\tau_W(F,G)(z)$ on $\textcolor{black}{J(z)\equiv}P(X+Y\leq z)$,  then no other copula can achieve this lower bound.

\end{theorem}
\begin{proof}

    Since 
    $\tau_W(F,G)(z-)=t$,
    by Theorem \ref{thm:optimal}\textcolor{black}{(i)}, copula $C_t$ achieves the bound $\tau_W(F,G)(z-)$ on $P(X+Y<z)$.
    That is, $C_t$ assigns mass $t$ to the set $\{(u,v)\subseteq [0,1]\times [0,1]:F^{-1}(u)+G^{-1}(v)<z\}$. Since, 
    \textcolor{black}{by hypothesis,}
    $C_t$ does not achieve the lower bound $\tau_W(F,G)(z)$ on $P(X+Y\leq z)$, $C_t$ assigns non-zero probability to the set $\{(u,v)\subseteq [0,1]\times [0,1]:F^{-1}(u)+G^{-1}(v)=z)\}$. In particular, the image of the set $\{(x,y):x+y=z\}$ under the $(F,G)$ mapping must contain a line segment with length greater than $0$ on the line $u+v-1=t$ in the $uv$\nobreakdash-plane inside the unit square\footnote{
    \textcolor{black}{The first two sentences of the proof imply that} when $\tau_W(F,G)(z)=\tau_W(F,G)(z-)$ this is 
    a necessary and sufficient condition for $C_t$ to assign non-zero probability to the set $\{(u,v)\subseteq [0,1]\times [0,1]:F^{-1}(u)+G^{-1}(v)=z)\}$.} as illustrated in Figure \ref{fig_a_panel}; since otherwise
    under $C_t$, $P(X+Y<z) = P(X+Y\leq z)$
    \textcolor{black}{in which case $C_t$ achieves the bound}.
    Let $a,b$ be such that the line segment $\{(u,v) : u=a+s,v=1-t-(a+s) \text{ for } s, 0\leq s\leq b-a \}$ 
     is contained in 
    $\{(u,v)\subseteq [0,1]\times [0,1]:F^{-1}(u)+G^{-1}(v)= z\}\cap \{(u,v)\subseteq [0,1]\times [0,1]:u+v-1=t\}$. 
    
    Now suppose there is a copula $C$ that achieves the lower bound $\tau_W(F,G)(z)$ on $P(X+Y\leq z)$.
    First, we claim that $C$ must assign mass $t$ to the rectangle $R_1=[0,a]\times [0,1+t-b]$. Since $R_1$ is a subset of $\{(u,v)\subseteq [0,1]\times [0,1]:F^{-1}(u)+G^{-1}(v)\leq z\}$ and by hypothesis $C$ achieves the lower bound, $C$ cannot assign mass more than $t$ to $R_1$. 
    
    Suppose $C$ assigns mass $0<r<t$ to $R_1$;
    \textcolor{black}{see Figure \ref{fig_b_panel}}.
    Note that we define the margins of the copula to be uniform ($C(p,1)=C(1,p)=p$, for all $p$ in $[0,1]$). In particular, in order for $C(1,1+t-b)=1+t-b$, $C$ needs to assign mass $1+t-b-r$ to $[a,1]\times [0,1+t-b]$ and similarly $C$ needs to assign mass $a-r$ to $[0,a]\times [1+t-b,1]$. As a consequence, $C$ needs to assign mass $1-(1+t-b-r)-(a-r)-r=b-a+r-t$ to the rectangle $[a,1]\times [1+t-b,1]$. 
    Now consider the rectangle $[a,b]\times[0,1+t-b]$. It needs to contain mass at least $t-r$ since $[a,b]\times[0,1]$ needs to contain mass $b-a$ and $[a,b]\times[1+t-b,1]\subseteq [a,1]\times[1+t-b,1]$. Similarly, $[0,a]\times [1+t-b,1+t-a]$ needs to contain mass at least $t-r$. Then $C$ assigns mass greater than or equal to $r+2(t-r)=t+(t-r)>t$ to $\{(u,v)\subseteq [0,1]\times [0,1]:F^{-1}(u)+G^{-1}(v)\leq z\}$, which is a contradiction that $C$ achieves the lower bound $\tau_W(F,G)(z)$ on $P(X+Y\leq z)$. Therefore, $C$ must assign mass $t$ to the rectangle $R_1=[0,a]\times [0,1+t-b]$.

Next, we show that $C$ assigns mass $b-a$ to the rectangle $[a,b]\times [1+t-b,1+t-a]$. By the hypothesis that $C$ achieves the lower bound $\tau_W(F,G)(z)$ on $P(X+Y\leq z)$,
$C$ assigns mass $1-t$ to the set $\{(u,v)\subseteq [0,1]\times [0,1]:F^{-1}(u)+G^{-1}(v)> z)\}$, which is a subset of the union of the following three rectangles: $[0,1]\times [1+t-a,1],[a,b]\times [1+t-b,1+t-a], [b,1]\times [0,1]$. In order to maintain uniform margins, the first and third rectangles contain mass $a-t$ and $1-b$. So the rectangle $[a,b]\times [1+t-b,1+t-a]$ needs to contain mass at least $(1-t)-(a-t)-(1-b)=b-a$. Again from the uniformity of the margins, $[a,b]\times [1+t-b,1+t-a]$ can contain mass at most $b-a$. Thus,  $C$ assigns mass $b-a$ to the rectangle $[a,b]\times [1+t-b,1+t-a]$. Furthermore, since the rectangle $[a,1]\times[1+t-b,1]$ contains total mass $1-(1+t-b)-(a-t)=b-a$, there's no mass elsewhere in this rectangle except in $[a,b]\times [1+t-b,1+t-a]$.

Now we show that $C$ needs to assign mass $b-a$ to the line segment $(a, 1+t-a)$ to $(b,1+t-b)$ inside the square $[a,b]\times [1+t-b,1+t-a]$. Figure \ref{fig_b_panel} shows a zoomed-in version of the rectangle $[a,b]\times [1+t-b,1+t-a]$. First, $C$ cannot assign any mass strictly below the line segment $(a, 1+t-a)$ to $(b,1+t-b)$ inside the square $[a,b]\times [1+t-b,1+t-a]$ because $C$ already assigns mass $t$ to $R_1$ and the total mass assigned by $C_t$ to the set $\{(u,v)\subseteq [0,1]\times [0,1]:F^{-1}(u)+G^{-1}(v)< z)\}$
is $t$. For any rectangle $[m,n]\times [c,d]$ such that $m+c-1\geq t$, $a\leq m< n\leq b, 1+t-b\leq c<d\leq 1+t-a$, suppose that $C$ assigns mass $\delta>0$ to this rectangle. We know that $C$ assigns mass $t$ to the region $R_1$. 
Let $E_1$ be the triangular area defined by vertices $(m,1+t-m),(a,1+t-a),(m,1+t-a)$
and $E_2$ be the triangular area defined by vertices $(m,1+t-m),(b,1+t-b),(b,1+t-m)$, 
\textcolor{black}{
as depicted in Figure \ref{fig_b_panel}}. Note that we have previously established that $C$ assigns no mass to the rectangle $[a,b]\times[1+t-a,1]$. The mass assigned by $C$ to $E_1$ and $R_1$ is equal to the mass in rectangle $[0,m]\times [0,1]$ subtracting the mass in the rectangle $[0,a]\times [1+t-a, 1]$, which is $C(m,1)-(a-t)=m-(a-t)=m-a+t$. Similarly, the mass assigned by $C$ to $E_2$ and $R_1$ is equal to $C(1,1+t-m)-(1-b)=1+t-m-(1-b)=t-m+b$. Thus, $C$ assigns the mass $m-a$ to $E_1$ and $b-m$ to $E_2$. Since $E_1, E_2$ are disjoint, $C$ assigns the mass $b-a$ to $E_1\cup E_2$. Then the rectangle $[m,n]\times [c,d]$ will contain mass $0$, which is a contradiction. Since the choice of $[m,n]\times [c,d]$ is arbitrary, we know that $C$ assigns mass $b-a$ to the line segment in $\mathbb{R}^2$
\textcolor{black}{from}
$(a, 1-t-a)$ to $(b,1-t-b)$ inside the square $[a,b]\times [1+t-b,1+t-a]$. Finally, $C$ assigns mass at least $t+(b-a)$ to the set $\{(u,v)\subseteq [0,1]\times [0,1]:F^{-1}(u)+G^{-1}(v)\leq z)\}$, which contradicts that $C$ achieves the lower bound $\tau_W(F,G)(z)$ on $P(X+Y\leq z)$. Thus, when $\tau_W(F,G)(z)=\tau_W(F,G)(z-)$, if the copula $C_t$ does not achieve the lower bound $\tau_W(F,G)(z)$ on $P(X+Y\leq z)$,  then no other copula can achieve this lower bound.

\end{proof}

\begin{figure}[htbp]
    \centering
    \begin{minipage}[b]{0.45\linewidth}
        \centering
\begin{tikzpicture}[x=0.75pt,y=0.75pt,yscale=-1,xscale=1]

\draw  (62.8,225.77) -- (297.8,225.77)(86.3,55.8) -- (86.3,244.65) (290.8,220.77) -- (297.8,225.77) -- (290.8,230.77) (81.3,62.8) -- (86.3,55.8) -- (91.3,62.8)  ;
\draw   (87,124.6) -- (187.8,124.6) -- (187.8,225) -- (87,225) -- cycle ;
\draw [color={rgb, 255:red, 255; green, 0; blue, 0 }  ,draw opacity=1 ]   (133,151.8) -- (145.5,165.75) ;
\draw    (87.5,206.5) -- (187,206.5) ;
\draw    (104,125) -- (104,224.5) ;
\draw  [dash pattern={on 0.84pt off 2.51pt}]  (133,151.8) -- (133.2,224.2) ;
\draw  [dash pattern={on 0.84pt off 2.51pt}]  (87,151.5) -- (133,151.8) ;
\draw  [dash pattern={on 0.84pt off 2.51pt}]  (87,167.5) -- (146.2,167.8) ;
\draw  [dash pattern={on 0.84pt off 2.51pt}]  (146.2,167.8) -- (146.6,224.2) ;
\draw [color={rgb, 255:red, 247; green, 5; blue, 5 }  ,draw opacity=1 ]   (145.5,165.75) .. controls (159.8,191.8) and (167.4,221.4) .. (169.4,225) ;
\draw [color={rgb, 255:red, 255; green, 0; blue, 0 }  ,draw opacity=1 ]   (87.4,135.8) .. controls (107.4,138.6) and (121.4,143.4) .. (133,151.8) ;
\draw  [dash pattern={on 0.84pt off 2.51pt}]  (146.2,167.8) -- (187,206.5) ;
\draw  [dash pattern={on 0.84pt off 2.51pt}]  (104,125) -- (133,151.8) ;
\draw  [fill={rgb, 255:red, 203; green, 85; blue, 85 }  ,fill opacity=1 ] (87,168.2) -- (132.6,168.2) -- (132.6,225) -- (87,225) -- cycle ;

\draw (182,229) node [anchor=north west][inner sep=0.75pt]   [align=left] {1};
\draw (73.5,118) node [anchor=north west][inner sep=0.75pt]   [align=left] {1};
\draw (88.9,227.85) node [anchor=north west][inner sep=0.75pt]   [align=left] {0};
\draw (294,228.4) node [anchor=north west][inner sep=0.75pt]    {$u$};
\draw (70,48.4) node [anchor=north west][inner sep=0.75pt]    {$v$};
\draw (101,228) node [anchor=north west][inner sep=0.75pt]  [font=\small] [align=left] {{\scriptsize $t$}};
\draw (73,197.3) node [anchor=north west][inner sep=0.75pt]  [font=\small] [align=left] {{\scriptsize $t$}};
\draw (127,230) node [anchor=north west][inner sep=0.75pt]  [font=\small] [align=left] {{\scriptsize $a$}};
\draw (145,228) node [anchor=north west][inner sep=0.75pt]  [font=\small] [align=left] {{\scriptsize $b$}};
\draw (45,161) node [anchor=north west][inner sep=0.75pt]   [align=left] {{\tiny $1+t-b$}};
\draw (45,145) node [anchor=north west][inner sep=0.75pt]   [align=left] {{\tiny $1+t-a$}};
\draw (99,186.2) node [anchor=north west][inner sep=0.75pt]    {$R_{1}$};
\end{tikzpicture}
        \caption{Copula $C$ and the image of $x+y=z$ under the $(F,G)$ mapping.}
        \label{fig_a_panel}
    \end{minipage}
    \hfill 
    \begin{minipage}[b]{0.45\linewidth}
        \centering

\tikzset{every picture/.style={line width=0.75pt}} 

\tikzset{every picture/.style={line width=0.75pt}} 

\tikzset{every picture/.style={line width=0.75pt}} 

\begin{tikzpicture}[x=0.65pt,y=0.65pt,yscale=-1,xscale=1]

\draw   (177.5,147.1) -- (263.3,147.1) -- (263.3,233.5) -- (177.5,233.5) -- cycle ;
\draw [color={rgb, 255:red, 255; green, 0; blue, 0 }  ,draw opacity=1 ]   (180.5,147.5) -- (263.5,229) ;
\draw  [fill={rgb, 255:red, 74; green, 144; blue, 226 }  ,fill opacity=1 ] (222.5,156) -- (249,156) -- (249,174.5) -- (222.5,174.5) -- cycle ;
\draw  [dash pattern={on 0.84pt off 2.51pt}]  (70,233.75) -- (177.5,233.5) ;
\draw  [dash pattern={on 0.84pt off 2.51pt}]  (223.9,189.1) -- (263.9,189.1) ;
\draw  [fill={rgb, 255:red, 233; green, 108; blue, 108 }  ,fill opacity=1 ] (222.5,147.5) -- (222.5,147.5) -- (222,188.25) -- (180.5,147.5) -- (180.5,147.5) -- cycle ;
\draw  [fill={rgb, 255:red, 233; green, 108; blue, 108 }  ,fill opacity=1 ] (264,188.25) -- (263.9,189.1) -- (263.5,229) -- (222,188.25) -- (222,188.25) -- cycle ;
\draw  [dash pattern={on 0.84pt off 2.51pt}]  (177.5,233.5) -- (178.5,357.75) ;
\draw  [dash pattern={on 0.84pt off 2.51pt}]  (263.3,233.5) -- (265,358.75) ;
\draw  [dash pattern={on 0.84pt off 2.51pt}]  (70,147.25) -- (177.5,147.1) ;
\draw  [fill={rgb, 255:red, 203; green, 85; blue, 85 }  ,fill opacity=1 ][dash pattern={on 0.84pt off 2.51pt}] (121.5,233.5) -- (177.5,233.5) -- (177.5,308.5) -- (121.5,308.5) -- cycle ;
\draw  [dash pattern={on 0.84pt off 2.51pt}]  (220.4,190.3) -- (222,358.25) ;
\draw  [dash pattern={on 0.84pt off 2.51pt}]  (70,190.25) -- (220.4,190.3) ;
\draw    (80.33,349.17) -- (121.5,308.5) ;
\draw    (173.67,308.33) -- (139.67,350.83) ;
\draw    (159,308) -- (120.33,351.17) ;
\draw    (142,308) -- (101,350.17) ;
\draw    (121,262.25) -- (77.33,301.83) ;
\draw    (121,248.5) -- (75.33,286.83) ;
\draw    (122,279.5) -- (80.33,316.83) ;
\draw    (178.33,320.33) -- (156,350.17) ;
\draw    (120.5,235.5) -- (76.67,270.5) ;
\draw    (121.67,294.33) -- (80.33,332.83) ;
\draw    (107.67,233) -- (77.67,255.17) ;

\draw (175,351.5) node [anchor=north west][inner sep=0.75pt]   [align=left] {{\scriptsize $a$}};
\draw (261.5,352.5) node [anchor=north west][inner sep=0.75pt]   [align=left] {{\scriptsize $b$}};
\draw (200.9,149.9) node [anchor=north west][inner sep=0.75pt]    {$E_{1}$};
\draw (241.3,191.9) node [anchor=north west][inner sep=0.75pt]    {$E_{2}$};
\draw (22.8,223.4) node [anchor=north west][inner sep=0.75pt]   [align=left] {{\scriptsize $1+t-b$}};
\draw (22.8,140.4) node [anchor=north west][inner sep=0.75pt]   [align=left] {{\scriptsize $1+t-a$}};
\draw (139.5,257.7) node [anchor=north west][inner sep=0.75pt]    {$R_{1}$};
\draw (217.5,352.5) node [anchor=north west][inner sep=0.75pt]   [align=left] {{\scriptsize $m$}};
\draw (22.8,181.4) node [anchor=north west][inner sep=0.75pt]   [align=left] {{\scriptsize $1+t-m$}};
\end{tikzpicture}
        \caption{A zoom in part of Figure \ref{fig_a_panel} with rectangle $[m,n]\times[c,d]$ colored in blue.}
        \label{fig_b_panel}
    \end{minipage}
\end{figure}

\begin{corollary}

            Given arbitrary $z$ and continuous $F,G$, let $t=\tau_W(F,G)(z)$. If the copula $C_t$ does not achieve the lower bound $\tau_W(F,G)(z)$ on $\textcolor{black}{J(z)\equiv}P(X+Y\leq z)$,  then no other copula can achieve this lower bound.
\end{corollary}
\begin{proof}

    This follows directly from Theorem \ref{thm:equi_C_t} 
    \textcolor{black}{since} when $F,G$ are both continuous, $\tau_W(F,G)(z)=\tau_W(F,G)(z-)$.
\end{proof}

\subsection{Sufficient conditions for achievability of the lower bound on $J(z)$
}
We will characterize another sufficient condition (different from Theorem \ref{thm:discrete_ach}) for the achievability of the lower bound on $J(z)=P(X+Y\leq z)$.
\begin{theorem}
            Given arbitrary $z$ and $F,G$, if $\tau_W(F,G)(z)>\tau_W(F,G)(z-)$ then the copula $C_t$ with $t=\tau_W(F,G)(z)$ will achieve the lower bound $\tau_W(F,G)(z)$ on $P(X+Y\leq z)$.\label{thm:lower_bound_ach_when_greater}
\end{theorem}
\begin{proof}

We will prove the contrapositive: if the lower bound $t=\tau_W(F,G)(z)$ of $P(X+Y\leq z)$ is not achievable under $C_t$, then $\tau_W(F,G)(z)=\tau_W(F,G)(z-)$. Note that $\tau_W(F,G)(z)=\tau_W(F,G)(z-)$ holds trivially when $\tau_W(F,G)(z)=0$. We will assume $t=\tau_W(F,G)(z)>0$.

If the lower bound $\tau_W(F,G)(z)$ of $P(X+Y\leq z)$ is not achievable under $C_t$, the set $\{(u,v)\in [0,1]^2 | F^{-1}(u)+G^{-1}(v)\leq z\}$ must contain a line segment with length greater than $0$ on the line $u+v-1=t$ in the $uv$\nobreakdash-plane. Based on Lemma \ref{lemma:klessz},
the set $\{(u,v)\in [0,1]^2 | F^{-1}(u)+G^{-1}(v)=z\}$ must contain a line segment with length greater than $0$ on the line $u+v-1=t$ in the $uv$\nobreakdash-plane. The existence of this line segment implies that the image of the set $\{(x,y):x+y=z\}$ under the $(F,G)$ mapping must also contain a line segment with length greater than $0$ on the line $u+v-1=t$ in the $uv$\nobreakdash-plane inside the unit square. 
This means that there exist $x^*$ and $\epsilon>0$ such that $F(x)+G(z-x)-1=t$ for all $x\in(x^*-\epsilon,x^*+\epsilon)$ and $F(\cdot)$ is continuous and strictly increasing on $x\in(x^*-\epsilon,x^*+\epsilon)$. In particular, for any $\delta>0$, there exists $\epsilon^*>0$ such that $F(x^*)-F(x^*-\epsilon^*)<\delta$. By definition of $x^*$, $\tau_W(F,G)(z)=F(x^*)+G(z-x^*)-1$.
Then for $\epsilon^*>0$,
\begin{align*}
    \tau_W(F,G)(z-\epsilon^*)&=\sup_{x+y=z-\epsilon^*}\max(F(x)+G(y)-1,0)\geq F(x^*-\epsilon^*)+G(z-x^*)-1.
\end{align*}
And
\begin{align*}
    \tau_W(F,G)(z)-\tau_W(F,G)(z-\epsilon^*)&=F(x^*)+G(z-x^*)-1- \tau_W(F,G)(z-\epsilon^*)\\
    &\leq F(x^*)+G(z-x^*)-1-(F(x^*-\epsilon^*)+G(z-x^*)-1)\\
    &=F(x^*)-F(x^*-\epsilon^*) <\delta.
\end{align*}

Since $\delta$ is arbitrary, $\tau_W(F,G)(\cdot)$ is continuous at $z$ and we must have $\tau_W(F,G)(z)=\tau_W(F,G)(z-)$.
\end{proof}

Thus, it follows from
Theorem \ref{thm:lower_bound_ach_when_greater}
that the {\em only} time when the lower bound on $P(X+Y\leq z)$ is {\em not} achievable is when the best possible point limits for $P(X+Y\leq z)$ and $P(X+Y<z)$ are the same. This result can be quite surprising: it follows that the distribution implied for $X+Y$ via the construction of $C_t$, namely, $\sigma_{C_t}(F, G)(z)=P(X+Y\leq z)$, is {\em discontinuous} at $z$ only when $\tau_W(F,G)(z)$ is {\em continuous} at $z$, i.e. $\tau_W(F,G)(z-)=\tau_W(F,G)(z)$.

We present an example to show that we do not require the margins $F,G$ to have uniform distributions on $[0,1]$  for the lower bound $\tau_W(F,G)(z)$ on $P(X+Y\leq z)$ to be not achievable.

\begin{exmp}
 Let

\[ F(x) = \begin{cases} 
      0 & x< 0, \\
      x^2 & 0\leq x < 1,\\
      1 & x\geq 1,
   \end{cases} \quad \quad  G(y) = \begin{cases} 
      0 & y< 0, \\
      1-(1-y)^2 & 0\leq y < 1,\\
      1 & y\geq 1.
   \end{cases}
\]
$F$ is the distribution for a random variable $X$ following a triangular distribution with $a=0, b=c=1$ (equivalent to Beta$(2,1)$), while $G$ is the distribution for a random variable $Y$ following a triangular distribution with $a=c=0,b=1$ (equivalent to Beta$(1,2)$). \par
Suppose $z=1$. Then 
\begin{align*}
    \tau_W(F,G)(1)&= \sup_{x+y=1}\max\{F(x)+G(y)-1,0\}\\
    &=\sup_{x}\max\{F(x)+G(1-x)-1,0\}\\
    &=0.
\end{align*}
The lower bound of $0$ corresponds to the copula $C_0$ constructed to achieve the lower  Fr\'{e}chet–Hoeffding bound (in other words, $X,Y$ are perfectly negatively correlated). In this example, $X=1-Y$. So under $C_0$, $P(X+Y= 1) = P(X+Y\leq 1)=1$ and $P(X+Y<1)=0$.\par
In contrast, it is not possible to construct a copula such that $P(X+Y\leq 1)=\tau_W(F,G)(1)=0$, so this lower bound is not achievable for $J(1) = P(X+Y\leq 1)$.
\end{exmp}

\subsection{Characterization of achievability of the lower bound on $J(z)$}
Theorem \ref{thm:when_achi} will provide necessary and sufficient conditions for the lower bound on $J(z)=P(X+Y\leq z)$ to be achievable. Theorem \ref{thm:sum_mak} provides a useful summary of the results concerning both the Makarov upper and lower bounds for $J(z) =P(X+Y\leq z)$.
\begin{theorem}
\label{thm:when_achi}
       The Makarov lower bound $t=\tau_W(F,G)(z)$ on $P(X+Y\leq z)$ is not achievable at $z$ if and only if there exist $x^*, y^*$ with $x^*+y^*=z$ such that all following three conditions hold: (i) $F(x^*)+G(y^*)= \sup_{x+y=z}\{F(x)+G(y)\}\geq 1$; (ii) $F(x)+G(z-x)$ is constant for $x$ in a neighborhood $Nr(x^*)$ of $x^*$; (iii) the image of the set $\{x,y:x\in Nr(x^*), y=z-x\}$ under the $(F,G)$ mapping contains an open interval within the line segment $\{(u,v)\in [0,1]^2 \mid  u+v-1=t\}$.\footnote{The conditions can also be defined similarly using the neighborhood of $y^*$.}
\end{theorem}
\begin{proof}

    If the Makarov lower bound $\tau_W(F,G)(z)$ on $P(X+Y\leq z)$ is not achievable at $z$, then Theorem \ref{thm:lower_bound_ach_when_greater} implies that we have to have 
$\tau_W(F,G)(z)=\tau_W(F,G)(z-)=t$. Furthermore, when the lower bound on $P(X+Y\leq z)$ is not achievable, it is not achievable under any copula, which includes $C_t$. Note that 
    by  (\ref{eq:frank-lower})
in Theorem \ref{thm:optimal},
under $C_t$, $P(X+Y\!<\!z) = t$.
Consequently,
    for $S_z:=\{(u,v)\in [0,1]^2 \mid  F^{-1}(u)+G^{-1}(v)=z\}$ and $H_1:=\{(u,v)\in [0,1]^2 \mid u+v-1=t\}$, $P(X+Y\!=\!z) = \iint_{S_z\cap H_1}dC_t(u,v)>0$, 
    since otherwise we would also have $P(X+Y\leq z)=t$ under $C_t$.
    
    If (i) does not hold, then 
    either $\sup_{x+y=z}\{F(x)+G(y)\}< 1$
    or $\sup_{x+y=z}\{F(x)+G(y)\}\geq 1$ but is not achieved on the set 
    $x+y=z$. Note that by the definition of $S_z$, for any $(u,v)\in S_z$, $u+v \leq \sup_{x+y=z}\{F(x)+G(y)\}$.
    If $\sup_{x+y=z}\{F(x)+G(y)\}<1$, then $t=\tau_W(F,G)(z)=0$ by (\ref{eq:low_valid}) 
    hence $H_1$ is the line $u+v=1$ and thus
    the set $S_z$ does not intersect $H_1$. 
    On the other hand, if $\sup_{x+y=z}\{F(x)+G(y)\}\geq 1$, then again by (\ref{eq:low_valid}), $1+t = \sup_{x+y=z}\{F(x)+G(y)\}$. If there do not exist $x^*, y^*$ such that $F(x^*)+G(y^*)= \sup_{x+y=z}\{F(x)+G(y)\}$, then for any $(u,v)\in S_z$, $u+v < \sup_{x+y=z}\{F(x)+G(y)\}$, and  thus $S_z\cap H_1 =\emptyset$ by definition of $H_1$. Hence in both sub-cases, the lower bound is achieved under $C_t$, which is a contradiction. 
    
    Now suppose that (i) holds. Since $C_t$ assigns mass uniformly to the set $H_1$, 
    and by hypothesis, 
     $\iint_{S_z\cap H_1}dC_t(u,v)>0$,
    there exists a line segment of positive length contained in
    $S_z\cap H_1$. Hence we may choose $u^*,v^*$ in the interior of this segment, such that there exist $x^*,y^*$ with 
    $ (F(x^*), G(y^*)) = (u^*,v^*)$
    and there is a  neighborhood of  $ (u^*,v^*)$ in $H_1$ that is contained in $S_z$, thus establishing (iii). (ii) also needs to hold by definition of $H_1$. 

    \medskip

    For the converse, if (i), (ii), (iii) hold, then $S_z$ contains a non-zero measure set in $H_1$. Since $C_t$ assigns mass uniformly to the set $H_1$, $\iint_{S_z\cap H_1}dC_t(u,v)>0$. This implies $P(X+Y=z)>0$ under $C_t$ and the lower bound $\tau_W(F,G)(z)$ will not be achievable at $z$ under $C_t$. Furthermore, by contrapositive of Theorem \ref{thm:lower_bound_ach_when_greater}, (i), (ii), (iii) imply that $\tau_W(F,G)(z)=\tau_W(F,G)(z-)$. Finally it then follows by Theorem \ref{thm:equi_C_t}, that since $C_t$ does not achieve the lower bound $\tau_W(F,G)(z)$ on $P(X+Y\leq z)$, this bound will not be achievable by any copula.
\end{proof}
\begin{theorem}\label{thm:sum_mak}
    The Makarov upper $\left(\rho_W(F,G)(z)\right)$
    and lower $(\tau_W(F,G)(z))$
    bounds on $P(X+Y\leq z)$ are pointwise best-possible,%
    \footnote{See Definition \ref{def:best_p}.}.
    The Makarov upper bound is achievable for each $z\in \mathbb{R}$, 
    but there may exist $z\in \mathbb{R}$ such that the
    lower bound 
    is not achievable.
  
    Conversely, the Makarov upper $\left(\rho_W(F,G)(z-)\right)$
    and lower $(\tau_W(F,G)(z-))$
    bounds on $P(X+Y< z)$ are pointwise best-possible.
    The Makarov lower bound is achievable for each $z\in \mathbb{R}$, 
    but there may exist $z\in \mathbb{R}$ such that the
    upper bound  is not achievable.
        
    The Makarov bounds are, in general, not uniformly sharp.\footnote{For given $F,G$, there does not exist a single joint distribution that achieves the bounds for all $z$. For example, the construction of the copula $C_t$ depends on the value $z$ in general.}
\end{theorem}
The following tables summarize when the Makarov bounds are always achievable (for all $z\in\mathbb{R}$ and any $F,G$) under different definitions of distribution functions.

\begin{table}[htbp]
  \centering
  \begin{tabular}{lcc}
    \toprule
    \multirow{2}{*}{\textbf{Bound}} & \multicolumn{2}{c}{$P(X + Y \leq z)$} \\
    \cmidrule(lr){2-3}
                  & Always Achievable & Pointwise Best-Possible \\
    \midrule
    Makarov Upper Bound $\rho_W(F,G)(z)$ & $\checkmark$ & $\checkmark$\\
    Makarov Lower Bound $\tau_W(F,G)(z)$ & \xmark  & $\checkmark$ \\
    \addlinespace
    \multirow{2}{*}{\textbf{Bound}} & \multicolumn{2}{c}{$P(X + Y < z)$} \\
    \cmidrule(lr){2-3}
                  & Always Achievable & Pointwise Best-Possible \\
    \midrule
    Makarov Upper Bound $\rho_W(F,G)(z-)$ & \xmark & $\checkmark$ \\
    Makarov Lower Bound $\tau_W(F,G)(z-)$ & $\checkmark$ & $\checkmark$\\
    \bottomrule
  \end{tabular}
    \caption{The Makarov upper bound $\rho_W(F,G)(z)$ on $P(X+Y\leq z)$ and the Makarov lower bound $\tau_W(F,G)(z-)$ on $P(X+Y< z)$ are always achievable for any given marginals $F,G$ and for all $z\in \mathbb{R}$. The achievabilities of the Makarov upper bound $\rho_W(F,G)(z-)$ on $P(X+Y<z)$ and the Makarov lower bound $\tau_W(F,G)(z)$ on $P(X+Y\leq z)$ are margin specific and depend on $z$. See Example \ref{ex:two} and Theorem \ref{thm:when_achi}. }
\end{table}

\subsection{Summary Theorems on Achievability and Sharpness}

We now add our new results to
the main result of \cite{frank1987best} given in Theorem \ref{thm:optimal}. Theorems \ref{thm:fns_exten_low} and \ref{thm:fns_exten_up} extend Theorem 3.2 in \cite{frank1987best} and provide a holistic picture of the achievability and sharpness of Makarov bounds on the sum of two random variables. 

Parts (i) of Theorems \ref{thm:fns_exten_low} and \ref{thm:fns_exten_up} follow directly from \cite{frank1987best} and \cite{nelsen2006introduction}; parts (ii) are stated in \cite{makarov1982estimates} and \cite{ruschendorf1982random} but not explicitly stated in \cite{frank1987best}; parts (iii) and (iv) follow from our new results in this section.
\begin{theorem}[Lower Makarov bounds]\label{thm:fns_exten_low}
    Let $F$ and $G$ be two fixed distribution functions. For any $z\in(-\infty, \infty)$:\\
(i) There exists a copula $C_{t^*}$\footnote{Previously defined in the proof of Theorem \ref{thm:discrete_ach}.}, dependent only on the value $t^*$ of $\tau_W(F, G)$ at $z-$, such that
\begin{align*}
    \sigma_{C_{t^*}}(F, G)(z-)=\tau_W(F, G)(z-)=t^*.
\end{align*}
In other words, the lower bound on $P(X+Y<z)$ is always achievable.

\noindent(ii) We have
\begin{align*}
    \tau_W(F,G)(z)=\inf_{C}\sigma_{C}(F,G)(z),
\end{align*}
where the infimum is taken over all copulas $C$. In other words, $t=\tau_W(F,G)(z)$ is the infimum of $P(X+Y\leq z)$ for all possible joint distributions of $X,Y$.

\noindent(iii) If $t^*=\tau_W(F,G)(z-)<t=\tau_W(F,G)(z)$, then there exists a copula $C_{t}$ dependent only on $t$ such that
\begin{align*}
    \sigma_{C_t}(F,G)(z)=\tau_W(F,G)(z).
\end{align*}
That is, the lower bound on $P(X+Y\leq z)$ can be achieved by $C_t$ when $\tau_W(F,G)(z-)<\tau_W(F,G)(z)$. We also have $\tau_W(F,G)(z)=\min_{C}\sigma_{C}(F,G)(z)$ where the minimun is taken over all copulas.

\noindent(iv) If there is no copula $C^*$ such that 
\begin{align*}
    \sigma_{C^*}(F,G)(z)=\tau_W(F,G)(z),
\end{align*}
then $\tau_W(F,G)(z)=\tau_W(F,G)(z-)$. (iv) is a contrapositive of (iii) which states that if the lower bound on $P(X+Y\leq z)$ is not achievable by any joint distribution, then the lower bounds for $P(X+Y\leq z)$ and $P(X+Y<z)$ must be the same.

\end{theorem}
Similarly, we can state the theorems for the upper bounds with a slight asymmetry.
\begin{theorem}[Upper Makarov bounds]\label{thm:fns_exten_up}
    Let $F$ and $G$ be two fixed distribution functions. For any $z\in(-\infty, \infty)$:\\
(i) There exists a copula $C_r$, dependent only on the value $r$ of $\rho_W(F, G)(z)$, such that
\begin{align*}
   \sigma_{C_r}(F, G)(z)=\rho_W(F, G)(z)=r. 
\end{align*}
In other words, the upper bound on $P(X+Y\leq z)$ is always achievable.

\noindent(ii) We have
\begin{align*}
    \rho_W(F,G)(z-)=\sup_{C}\sigma_{C}(F,G)(z-),
\end{align*}
where the supremum is taken over all copulas $C$. In other words, $r^*=\rho_W(F,G)(z-)$ is the supremum of $P(X+Y< z)$ for all possible joint distributions of $X,Y$.

\noindent(iii) If $r^*=\rho_W(F,G)(z-)<r=\rho_W(F,G)(z)$, then there exists a copula $C_{r*}$ dependent only on $r^*$ such that
\begin{align*}
    \sigma_{C_{r^*}}(F,G)(z-)=\rho_W(F,G)(z-).
\end{align*}
That is, the upper bound on $P(X+Y< z)$ can be achieved by $C_{r^*}$ when $\rho_W(F,G)(z-)<\rho_W(F,G)(z)$. We also have $\rho_W(F,G)(z)=\max_{C}\sigma_{C}(F,G)(z)$ where the maximum is taken over all copulas.

\noindent(iv) If there is no copula $C^*$ such that 
\begin{align*}
    \sigma_{C^*}(F,G)(z-)=\rho_W(F,G)(z-),
\end{align*}
then $\rho_W(F,G)(z)=\rho_W(F,G)(z-)$. (iv) is a contrapositive of (iii) which states that if the upper bound on $P(X+Y< z)$ is not achievable by any joint distribution, then the upper bounds for $P(X+Y\leq z)$ and $P(X+Y<z)$ must be the same.

\end{theorem}

\subsection{Examples Revisited}

We now illustrate these results by revisiting Examples \ref{ex:one} and \ref{ex:two}. Figure \ref{fig:summary_ex11_12} summarizes the achievability of the Makarov upper and lower bounds in the two examples.


\pgfplotsset{
    soldot/.style={color=blue,only marks,mark=*}, 
    holdot/.style={color=blue,fill=white,only marks,mark=*} 
}
\tikzset{every picture/.style={line width=0.75pt}} 
\begin{figure}[h!]
\centering
\includegraphics[width=\linewidth]{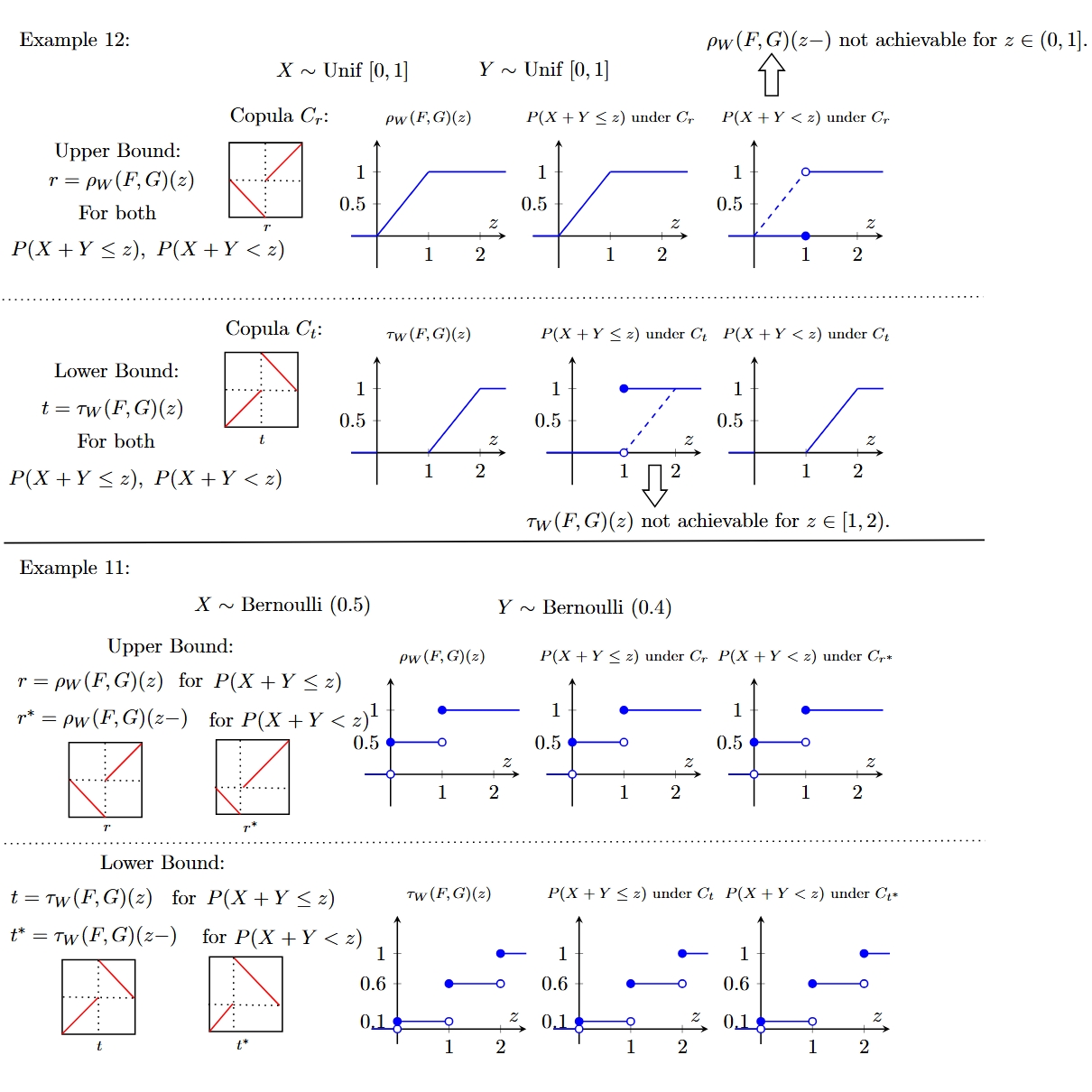}
\caption{In the right two panels of Example \ref{ex:two}, the dashed lines show the bounds $\rho_W(F,G)(z)$ and $\tau_W(F,G)(z)$ and the solid lines show what is achieved under the copulas; the difference indicates the bounds are not achievable. Hence only the upper bound on $P(X+Y\leq z)$ and the lower bound on $P(X+Y\leq  z)$ can be achieved for all $z$. The upper and lower bounds for $P(X+Y< z)$ and $P(X+Y<z)$ are the same, which follows from part (iv) of Theorem \ref{thm:fns_exten_low} and \ref{thm:fns_exten_up}. In Example \ref{ex:one}, the upper and lower bounds for $P(X+Y\leq z)$ and $P(X+Y<z)$ are different, all bounds on $P(X+Y\leq z)$ and $P(X+Y<z)$ are achieved but under different copula constructions.}
\label{fig:summary_ex11_12}
\end{figure}

\cite{makarov1982estimates} introduced the supremum of $P(X + Y \leq z)$, and \cite{frank1987best} showed that the copula $C_r$, where $r = \rho_{W}(F, G)(z)$, attains this supremum. This naturally leads to the association of $C_r$ with the supremum of $P(X + Y < z)$, especially when $F$ and $G$ are continuous, making the suprema of  $P(X + Y \leq z)$ and $P(X + Y < z)$ identical.
However, it is important to note that the supremum for $P(X + Y < z)$ is not always achievable, and the value obtained using the copula  $C_r$—where $r = \rho_{W}(F, G)(z-) = \rho_{W}(F, G)(z)$—can differ from the true (unachievable) supremum by as much as $1$. In other words, using the copula $C_r$ to estimate the upper bound for $P(X + Y < z)$ is suboptimal and may significantly deviate from the optimal bound.
To see this consider Example \ref{ex:two} displayed in the top panel of Figure \ref{fig:summary_ex11_12}. Here we see that when $z=1$, $r=\rho_{W}(F,G)(z-)=\rho_{W}(F,G)(z)=1$. However, in this case $C_{r}$ corresponds to the degenerate joint distribution under which $X+Y=1$, implying that $P(X+Y<z)=0$. In other words, in this example the Makarov upper bound on $P(X+Y<1)$ is $1$, but the copula $C_r$ which achieves the upper bound on $P(X+Y \leq 1)$ has
$P(X+Y< 1)=0$ (!) which is as different as it is possible to be from the supremum value. 
More generally, in this example, for $z\in (0,1]$, the supremum for $P(X+Y<z)$ is $z$, yet under $C_r = C_z$ we have $P(X+Y<z) =0$.

Note that although the copula $C_r$ (entirely) fails to achieve the upper bound on $P(X+Y<z)$, we may use another copula $C_{r^*}$ where $r^*<r$ to obtain a joint distribution under which $P(X+Y<z)$ is arbitrarily close to the unattainable supremum. Specifically, in this example,  
the copula $C_{1-\epsilon}$, for some small $\epsilon>0$, gives a joint distribution under which 
$P(X+Y < 1)$ is arbitrarily close to the (unattainable) upper bound of $1$. 
This is because under $C_{1-\epsilon}$, 
 $1-\epsilon = P(X+Y= 1-\epsilon) =
 P(X+Y\leq 1-\epsilon) < P(X+Y < 1)$.

Conversely, the difference between the (unachievable) infimum of the (achievable) 
lower bounds on $P(X+Y\leq z)$ and the value achieved under the copula $C_t$ where $t=\tau_{W}(F,G)(z)$ can also be $1$.



\section{Sharp bounds on the difference}
\label{sec:frank}

Now we consider the closely related problem concerning bounds on the distribution of the difference of two random variables with fixed marginals. 

As we will describe in the next section \ref{sec:fan} below, bounds on the distribution function of the difference of two random variables with fixed marginals have been widely used for partial identification of individual causal effects (see \citealt{fan2010sharp}, \citealt{imbens2018causal}, \citealt{firpo2019partial}). 

In this section, we will first focus on the theoretical results. Specifically, we will apply our previous results concerning best possible bounds on sums, but replacing one variable by its negation. Although these results are corollaries of the previous results, we include them here to connect various results in the literature and for the convenience of readers.

Let $X$ and $Y$ be random variables with respective distribution functions $F$ and $G$ fixed. Let $\Delta=X-Y$ be the difference of random variables $X,Y$. Let $J_\Delta(\cdot)$ be the distribution function of $\Delta$.
\begin{theorem}\label{bounds}
For any given value $\delta$, best-possible bounds $\underline{J_\Delta}(\delta)\leq J_\Delta(\delta)\leq \overline{J_\Delta}(\delta)$
on $J_\Delta(\delta)$ are given by
 \begin{align*}
\underline{J_\Delta}(\delta)&=\sup_{x-y=\delta}\max\{F(x)-P(Y<y),0\}\nonumber \\
&= \sup_{x-y=\delta}\max\{F(x)-G(y)+P(Y=y),0\};\\[16pt]
\overline{J_\Delta}(\delta)&=1+\inf_{x-y=\delta} \min\{F(x)-P(Y<y),0\}\nonumber \\
&=  1+\inf_{x-y=\delta} \min\{F(x)-G(y)+P(Y=y),0\}.
\end{align*}

\end{theorem}

The bounds in Theorem \ref{bounds} differ from those stated in \cite{fan2010sharp} and \cite{williamson1990probabilistic} by the inclusion of the point mass $P(Y=y)$ term in both the upper and lower bounds. As a consequence, the lower bound reported by these authors is not best possible; see
Theorem \ref{upper:equal} and comment below  Theorem \ref{upper:equal} for more details.

\subsection{Proof of Theorem \ref{bounds}}
The proof is a direct application of Theorem \ref{Thm:franks}. Consider a new variable $Y'=-Y$ with cdf $G'$. Then from equation (\ref{eq:low_valid}) and (\ref{eq:up_valid}), for any $\delta$, the bound on $P(\Delta\leq \delta)=P(X-Y\leq \delta)=P(X+Y'\leq \delta)$ is:
\begin{align*}
\underline{J_\Delta}(\delta)&=\sup_{x+y'=\delta}\max(F(x)+G'(y')-1,0), \\ 
\overline{J_\Delta}(\delta)&=1+\inf_{x+y'=\delta} \min(F(x)+G'(y')-1,0). 
\end{align*}
Note that 
\begin{align*}G'(y')&=P(-Y\leq y')\\
&=P(Y\geq -y')\\
&=1-P(Y<-y' )\\
&=1-G(-y') + P(Y  = -y').\end{align*}
Replace $y'$ with $-y$ and $G'(y')$ with $1-G(y)+P(Y=y)$, we get the best-possible bounds in Theorem $\ref{bounds}$. \\

\begin{remark}[]
In the case where $P(Y=y)=0$ for all $y$ (for example, when the distribution function of $Y$ is absolutely continuous with respect to the Lebesgue measure), we recover the best-possible bounds in Theorem $2$ of \cite{williamson1990probabilistic} and in Lemma $2.1$ of \cite{fan2010sharp}. However,  when $G$ is not absolutely continuous, the bounds in \cite{fan2010sharp}  and \cite{williamson1990probabilistic} can be different from the bounds in Theorem \ref{bounds} as the point mass $P(Y=y)$ can be nonzero for some $y\in\mathbb{R}$. In the Appendix \ref{App:A}, we identify the issue in the proof of Theorem 2 in \cite{williamson1990probabilistic}. Though we didn't address it explicitly, the same issue applies to the \cite{williamson1990probabilistic} bounds on the ratio of two random variables.
\end{remark}

\subsection{Implications of the new bounds}

At this point, Theorem \ref{bounds} seems to imply that the lower bound in \cite{williamson1990probabilistic} is valid but not necessarily best-possible and the upper bound might not be valid. The next theorem will establish that in fact of the upper bound on the cdf for the difference $\Delta=X-Y$ in \cite{williamson1990probabilistic} is valid even though the proof used in \cite{williamson1990probabilistic} is not correct.
\begin{theorem}\label{upper:equal}
For any random variables $X$ and $Y$ with respective cdfs $F(\cdot)$ and $G(\cdot)$,
\begin{align*}
    \inf_{x-y=\delta} \min\{F(x)-P(Y<y),0\}
&=  \inf_{x-y=\delta} \min\{F(x)-G(y),0\}.
\end{align*}
\end{theorem}
The proof of Theorem \ref{upper:equal} is left to Appendix \ref{App:B}. Theorem \ref{upper:equal} implies that the upper bounds in \cite{williamson1990probabilistic} and \cite{fan2010sharp} coincide with the upper bounds we proposed in Theorem \ref{bounds}. Since the lower bounds that we propose in Theorem \ref{bounds} are greater than or equal to the lower bounds in \cite{williamson1990probabilistic}, \cite{fan2010sharp}, Theorem $\ref{upper:equal}$ establishes that somewhat surprisingly all these bounds are valid, though the lower bounds are not sharp. In Section \ref{subsec:app_CDF}, we demonstrate through an example that the lower bound proposed by \cite{williamson1990probabilistic} and \cite{fan2010sharp} is not sharp.

\begin{remark}
 Theorem $3$ in \cite{williamson1990probabilistic} states an  optimality result (analogous to Theorem \ref{thm:optimal} of our paper) contains an unnecessary exclusion: specifically, it states that the bounds are only achievable if $F$ and $G$ are not both discontinuous at some $x,y$ such that $x+y=z$. It appears that this additional unnecessary condition was added because Williamson and Downs fail to note that when $F$ and $G$ are discontinuous the bounds only take a strict subset of values in $[0,1]$. Consequently, only a subset of values need to be considered. 
\end{remark}

\section{A causal perspective}
\label{sec:fan}
Throughout this section, we consider a binary treatment $D=0,1$. Let $Y_{1}$ be the potential outcome were an individual to take the treatment and $Y_{0}$ be the potential outcome were an individual to take control. We assume the stable unit
treatment value assumption (SUTVA, \citealt{rubin1978bayesian}) that there is a single version of each treatment/control and no interference among the subjects. We define our parameter of interest  as the individual treatment effect: $\Delta=Y_1-Y_0$. \cite{fan2010sharp} Lemma $2.1$ stated the bounds on the distribution function of the individual treatment effect and claimed that they are sharp. We modify the bounds in \cite{fan2010sharp} based on Theorem \ref{bounds}. Let $F_1, F_0$ be the cumulative distribution function on $Y_1, Y_0$ respectively. Let $F_\Delta(\cdot)$ be the cdf for $\Delta$.
\begin{theorem}\label{causal_bounds}
For any given value $\delta$, best-possible bounds on $F_\Delta(\delta)$ are given by
 \begin{align}
F^L(\delta)&=\sup_y\max\{F_1(y)-P(Y_0<y-\delta),0\}\nonumber \\
&= \sup_y\max\{F_1(y)-F_0(y-\delta)+P(Y_0=y-\delta),0\}; \label{eq:causal_correct_lower}\\[16pt]
F^U(\delta)&=1+\inf_y \min\{F_1(y)-P(Y_0<y-\delta),0\}\nonumber \\
&=  1+\inf_y \min\{F_1(y)-F_0(y-\delta)+P(Y_0=y-\delta),0\}.\label{eq:causal_correct_upper}
\end{align}
\end{theorem}
Let $Y$ denote the observed outcome variable. Under consistency, $Y=Y_{0}$ when $D=0$ and $Y=Y_1$ when $D=1$. In practice, if we are willing to assume ignorability (for example, in randomized clinical trials
(RCTs)) or conditional ignorability, the marginal distributions $F_1(y)$ and $F_0(y)$ can be identified. Theorem \ref{causal_bounds} allows us to conclude best-possible bounds on the distribution function of the individual treatment effect. In the special case where $Y$ is ordinal, Proposition $1$ in \cite{lu2018treatment} can be recovered using Theorem \ref{bounds} and Theorem \ref{upper:equal}. \cite{lu2018treatment} consider a special case where $Y$ is non-negative and prove the bounds using a construction argument instead of the copula theory.
\begin{corollary}
The Fan-Park upper bound is best-possible.
\end{corollary}

\noindent{\it Proof:} This follows directly from Theorem \ref{upper:equal} where $X$ is replace by $Y_1$ and $Y$ is replace by $Y_0$.\hfill $\Box$

\subsection{Application of Theorem $\ref{causal_bounds}$ on cdf bounds of ITE}\label{subsec:app_CDF}
Here we will present a simple example that applies our bounds in Theorem $\ref{causal_bounds}$ and compare them with the bounds in \cite{fan2010sharp}. 

Consider the case where we have a binary treatment variable $(D=0, 1)$ and a ternary response $(Y=0,1,2)$. Under randomization, the relationship between the counterfactual distribution
 $P(Y_0, Y_1)$ and the observed distributions $\left\{P(Y\mid D=0), P(Y\mid D=1)\right\}$ 
  is given by:
 $P(Y\!=\!i\mid D\!=\!j) = P( Y_j\!=\!i)$. Suppose we observe the marginals given in Table $\ref{example_table}$. We can parameterize the joint distribution with $4$ parameters $p,q,t,r$. Note that by the Fr\'{e}chet inequalities, $\max\{ P(Y_0=i)+P(Y_1=j)-1,0\}\leq P(Y_0=i,Y_1=j)\leq \min\{P(Y_0=i), P(Y_1=j)\}$. We get ranges for $p,q,t,r$ by applying Fr\'{e}chet inequalities to each of the quantity.

\begin{center}
\begin{table}[h]
\scalebox{0.75}{
\begin{tabular}{r|ccc}
 & $P(Y\!=\!0\mid D\!=\!0)=0.3$ & $P(Y\!=\!1\mid D\!=\!0)=0.2$ & $P(Y\!=\!2\mid D\!=\!0)=0.5$\\[4pt]
 \hline
 &\\[-5pt]
 $P(Y\!=\!0\mid D\!=\!1)=0.7$ & $P(Y=0\mid D=1)-p-r$ & $p\in[0,0.2]$ & $r\in[0.2,0.5]$\\[4pt]
  $P(Y\!=\!1\mid D\!=\!1)=0.1$ & $P(Y=1\mid D=1)-t-q$ &$t\in[0,0.1]$ &$q\in[0,0.1]$ \\[4pt]
    $P(Y\!=\!2\mid D\!=\!1)=0.2$ & $1-(\hdots)$ &$P(Y=1\mid D=0)-t-p$ &$P(Y=2\mid D=0)-r-q$\\[4pt]
\end{tabular}
}
\caption{Application with binary treatment and ternary outcome}
\label{example_table}
\end{table}
\end{center}

Based on the bounds proposed in (\ref{eq:causal_correct_lower}) and (\ref{eq:causal_correct_upper}) , we note the following alternative expressions for $F^{L}(\delta)$ and $F^{U}(\delta)$ :
$$
\begin{aligned}
&F^{L}(\delta)=\max \left(\sup _{y}\left\{F_{1}(y)-P(Y_0<y-\delta)\right\}, 0\right); \\
&F^{U}(\delta)=1+\min \left(\inf _{y}\left\{F_{1}(y)-P(Y_0<y-\delta)\right\}, 0\right).
\end{aligned}
$$

Note that $\delta=-2$ is only possible when $Y_1=0, Y_0=2$. So this corresponds to the entry in the top right corner of Table $\ref{example_table}$. By the Fr\'{e}chet inequalities, the bounds on $P(Y_1=0, Y_0=2)$ is given by $r\in[0.2,0.5]$. Now consider $F_1(y)-P(Y_0<y-\delta)$ in our example, 

$$F_1(y)-P(Y_0<y+2)=\begin{cases}0 & y\leq-2\\-0.3 &-2< y\leq-1,\\
-0.5 &-1< y<0,\\
0.2 & y=0,\\
-0.3&0< y<1,\\
-0.2 & 1\leq y<2,\\
0 & y\geq 2.\end{cases}$$
This gives $F_\Delta(-2)\in[0.2,0.5]$, which matches the range given by Fr\'{e}chet inequalities. In this case if we consider the bounds proposed in Lemma 2.1 in \cite{fan2010sharp}, which follow from Theorem 2 of \cite{williamson1990probabilistic}, 
$$F_1(y)-F_0(y+2)=\begin{cases}0 & y<-2,\\
-0.3 &-2\leq y<-1,\\
-0.5 &-1\leq y< 0,\\
-0.3&0\leq y<1,\\
-0.2 & 1\leq y<2,\\
0 & y\geq 2.\end{cases}$$
The lower bounds for $F_\Delta(-2)$ is $0$, which is not sharp. This example corresponds to the case where $F_1$ and $F_0$ are both discontinuous at $Y_1=0$ and $Y_0=2$.

As a second illustration of Theorem \ref{causal_bounds} in a context where the Fr\'{e}chet inequalities are not directly applicable, consider the bounds for $F_\Delta(-1)$ given by Theorem \ref{causal_bounds}. The individual treatment effect is less than or equal to $-1$ when $Y_1 = 0, Y_0=1$ or $Y_1 = 1, Y_0 =2$ or $Y_1 =0, Y_0 =2$. Therefore, based on Table \ref{example_table}, $F_\Delta(-1) = p+q+r$. To calculate the bounds, consider $F_1(y)-P(Y_0<y+1)$ in our example, 

$$F_1(y)-P(Y_0<y+1)=\begin{cases}0 & y\leq-1,\\
-0.3 &-1< y<0,\\
0.4&y=0,\\
0.2&0< y<1,\\
0.3&y=1,\\
-0.2 & 1< y<2,\\
0 & y\geq 2.\end{cases}$$

This gives bounds for $F_\Delta(-1)=p+q+r\in[0.4,0.7]$. Again in this case, if we compute the lower bound based on \cite{fan2010sharp} and \cite{williamson1990probabilistic}, it is not sharp. Finally, we complete the example by obtaining the bounds $F_\Delta(0)=1.2+t-r\in[0.7,1]$, $F_\Delta(1)=1.5-p-q-t-r\in[0.8,1]$. $F_\Delta(2)=1$ follows trivially by construction.

Inference on Makarov type bounds used in causal inference can be found in \cite{fan2010sharp} and is discussed in \cite{fan2012confidence} and \cite{imbens2018causal}.

\section{Acknowledgments}
We thank Carlos Cinelli, Yanqin Fan, Giovanni Puccetti, Ludger Rüschendorf, James M. Robins  and the audience of the UW causal reading group as well as the 2024 ACIC poster session for valuable input and discussion.

\newpage
\singlespacing
\bibliography{Bounds_on_the_Distribution_of_a_Sum_of_Two_Random_Variables_arxiv.bib}

\newpage
\doublespacing
\begin{appendix}

\section{Table on prior work and relation to this paper} \label{app:table_prior}
\begin{table}[h]



\begin{tabular}{l c c c c c c}
\hline
& {Def.~of} & {Def.~of} & \multicolumn{2}{c}{Lower Bound on} & \multicolumn{2}{c}{Upper Bound on} \\
& cdf & Sharpness & $P(Z < z)$ & $P(Z \leq z)$ & $P(Z < z)$ & $P(Z \leq z)$ \\
\hline
\hline
\\[-4pt]
\cite{makarov1982estimates} & $\widetilde{F}$ & $\sup/\inf$  & \checkmark & \checkmark & \checkmark & \checkmark \\[2pt]
\cite{ruschendorf1982random}  & $F$ & attainable & \checkmark   &  & & \checkmark \\[2pt]
\cite{frank1987best}& $\widetilde{F}$ & attainable & \checkmark & &  & \checkmark \\[2pt]
\cite{nelsen2006introduction}   & $F$ & attainable & \checkmark & &  & \checkmark \\[6pt]
\multirow{2}{*}{This paper}    & \multirow{2}{*}{$F$} & sup/inf & \checkmark & \checkmark  & \checkmark & \checkmark \\
&& attainable & \checkmark &
Thm.~\ref{thm:fns_exten_low} & Thm.~\ref{thm:fns_exten_up} & \checkmark\\
\hline
\end{tabular}

\caption{Concordance with prior work illustrating the different definitions and bounds provided in relation to the results in this paper. ‘sup/inf’ indicates that the given bound is the best‐possible, i.e. sup or inf over all joint distributions consistent with the prescribed marginals, whereas ‘attainable’ means there is at least one joint distribution achieving that bound. A $\checkmark$ means the property (best‐possible or attainability) holds for all choices of marginals, a blank means it may fail for some marginals, and Theorems~\ref{thm:fns_exten_low}-\ref{thm:fns_exten_up} provide conditions ensuring these bounds are indeed attained.\label{tab:literature}}    
\end{table}

\section{Relationship between bounds defined using left or right continuous cdfs}\label{app:relation_left_right_cdf_bounds}
Note that we previously defined the left‐continuous distribution functions 
$\widetilde{F}(x) := P(X < x)$ and $\widetilde{G}(y) := P(Y < y)$, 
so that indeed $\widetilde{F}(x)=F(x-)$ and $\widetilde{G}(x)=G(x-)$. 
Recall that the lower bound in our setup is given by
\[
  \tau_W(F,G)(z)\;=\;\sup_{x+y=z}\max\bigl(F(x)+G(y)-1,\,0\bigr),
\]
and the upper bound is
\[
  \rho_W(F,G)(z)\;=\;1+\inf_{x+y=z}\min\bigl(0,\;F(x)+G(y)-1\bigr).
\]
We want to prove that for fixed $z$,
\[
  \tau_W(\widetilde{F},\widetilde{G})(z-)\;\le\; 
  \tau_W(\widetilde{F},\widetilde{G})(z)\;=\;
  \tau_W(F,G)(z-)\;\le\;
  \tau_W(F,G)(z),
\]
and
\[
  \rho_W(\widetilde{F},\widetilde{G})(z-)\;\le\; 
  \rho_W(\widetilde{F},\widetilde{G})(z)\;=\;
  \rho_W(F,G)(z-)\;\le\;
  \rho_W(F,G)(z).
\]
The outer inequalities are given by monotonicity arguments. In particular, for fixed $F,G$, the map $z \mapsto \tau_W(F,G)(z)$ is non-decreasing, so $\tau_W(F,G)(z-) \le \tau_W(F,G)(z)$, and similarly for $\tau_W(\widetilde{F},\widetilde{G})$, $\rho_W(F,G)$, and $\rho_W(\widetilde{F},\widetilde{G})$. It remains to show 
\[
  \tau_W(\widetilde{F},\widetilde{G})(z) = \tau_W(F,G)(z-)
  \quad\text{and}\quad 
  \rho_W(\widetilde{F},\widetilde{G})(z) = \rho_W(F,G)(z-).
\] 
First, we focus on $\tau_W(\widetilde{F},\widetilde{G})(z)= \tau_W(F,G)(z-)$. Based on the definition of $\tau_W$, we want to show that
\[\sup_{x+y=z}\max(\widetilde{F}(x)+\widetilde{G}(y)-1,0)=\sup_{x+y=z}\max(F(x-)+G(y-)-1,0)=\sup_{x+y=z-}\max(F(x)+G(y)-1,0)\]
Because adding a constant and taking max with $0$ does not change equality, it suffices to show
\[\sup_{x+y=z}(F(x-)+G(y-))=\sup_{x+y=z-}(F(x)+G(y)).\]
We will prove this equality by showing two inequalities. First, we have
\begin{align}
    \sup_{x+y=z-}(F(x)+G(y)) & := \lim_{h^*\to 0, h^*>0}\sup_{x+y=z-h^*}(F(x)+G(y))\\
    & = \lim_{h^*\to 0, h^*>0}\sup_{x+y=z}(F(x-\frac{h^*}{2})+G(y-\frac{h^*}{2})) .\label{eq:lim_h_sup_z}
\end{align}
Since $F,G$ are non-decreasing, for all $x,y$ and $h^*>0$,
\[F(x-\frac{h^*}{2})+G(y-\frac{h^*}{2})\leq \lim_{h\to 0, h>0}(F(x-\frac{h}{2})+G(y-\frac{h}{2})).\]
Taking the sup on both sides gives:
\[\sup_{x+y=z}F(x-\frac{h^*}{2})+G(y-\frac{h^*}{2})\leq \sup_{x+y=z}\lim_{h\to 0, h>0}(F(x-\frac{h}{2})+G(y-\frac{h}{2})).\]
Therefore, 
\begin{align*}
    \sup_{x+y=z-}(F(x)+G(y)) 
     = \lim_{h^*\to 0, h^*>0}\sup_{x+y=z}(F(x-\frac{h^*}{2})+G(y-\frac{h^*}{2})) 
      \leq  \sup_{x+y=z}(F(x-)+G(y-))
\end{align*}
For the second direction, 
  since \(F\) and \(G\) are non-decreasing, for any \(\epsilon > 0\), there exists \(h^* > 0\) such that for all $h<h^*$:
  \begin{align*}
        F(x - \frac{h}{2}) \geq F(x-) - \frac{\epsilon}{2}, \quad G(y - \frac{h}{2}) \geq G(y-) - \frac{\epsilon}{2}.
  \end{align*}
  This implies for all $h<h^*$:
  \begin{align}
  F(x - \frac{h}{2}) + G(y - \frac{h}{2}) \geq F(x-) + G(y-) - \epsilon.\label{eq:-h_greater_epsilon}
  \end{align}
  Taking the supremum over \((x, y)\) gives:
  \[
  \sup_{x+y=z} (F(x - \frac{h}{2}) + G(y - \frac{h}{2})) \geq \sup_{{x+y=z}} (F(x-) + G(y-)) - \epsilon
  \]
  for all $h<h^*$. The \(\lim_{h \downarrow 0}\) of the left-hand side (which is equal to (\ref{eq:lim_h_sup_z})) is at least the right-hand side. Since $\epsilon>0$ was arbitrary, this gives
  \[\sup_{x+y=z-}(F(x)+G(y))\geq \sup_{{x+y=z}} (F(x-) + G(y-)).\]
Therefore, we have
$$\sup_{x+y=z-}(F(x)+G(y))=\sup_{x+y=z}(F(x-)+G(y-)).$$

A completely analogous argument applies to the upper bound $\rho_W$, just replacing the `$\sup$' by `$\inf$'. We want to prove that 
\[\inf_{x+y=z}\min(1,F(x-)+G(y-))=\inf_{x+y=z-}\min(1,F(x)+G(y)).\]
Again it suffices to show that
\[\inf_{x+y=z}(F(x-)+G(y-))=\inf_{x+y=z-}(F(x)+G(y)).\]
Since $F,G$ are non-decreasing, we note that 
\[\inf_{x+y=z}F(x-\frac{h^*}{2})+G(y-\frac{h^*}{2})\leq \inf_{x+y=z}\lim_{h\to 0, h>0}(F(x-\frac{h}{2})+G(y-\frac{h}{2})).\]
Therefore, 
\begin{align*}
    \inf_{x+y=z-}(F(x)+G(y)) 
     = \lim_{h^*\to 0, h^*>0}\inf_{x+y=z}(F(x-\frac{h^*}{2})+G(y-\frac{h^*}{2})) 
      \leq  \inf_{x+y=z}(F(x-)+G(y-)).
\end{align*}
Using equation (\ref{eq:-h_greater_epsilon}), and taking the infimum over \((x, y)\) gives:
  \[
  \inf_{x+y=z} (F(x - \frac{h}{2}) + G(y - \frac{h}{2})) \geq \inf_{{x+y=z}} (F(x-) + G(y-)) - \epsilon
  \]
  for all $h<h^*$. The \(\lim_{h \downarrow 0}\) of the left-hand side is at least the right-hand side. Since $\epsilon>0$ was arbitrary, this gives
  \[\inf_{x+y=z-}(F(x)+G(y))\geq \inf_{{x+y=z}} (F(x-) + G(y-)),\]
and
$$\inf_{x+y=z-}(F(x)+G(y))=\inf_{x+y=z}(F(x-)+G(y-)).$$

\section{R{\"u}schendorf's approach to achievability of the bounds}\label{App:C}

\cite{puccetti2012computation} comment that a general result in \cite{ruschendorf1983solution} implies the achievability of the infimum and supremum of functionals (lower bound on $P(X+Y<z)$ and upper bound on $P(X+Y\leq z)$). Here for the sake of completeness, we give the argument in detail. This provides a (non-constructive) proof of the result in \cite{frank1987best}. To prove the achievability of the bounds, we start with a definition.
\begin{definition}
    A function \( \varphi: \mathbb{R}^2 \to \mathbb{R} \) is lower semicontinuous if, for all \( (x,y) \in \mathbb{R}^2 \),
\[
\liminf_{(x',y') \to (x,y)} \varphi(x', y') \geq \varphi(x, y).
\]
\end{definition}
Given marginal distribution functions \( F \) and \( G \) for random variables \( X \) and \( Y \) respectively, let \( \mathcal{M}(F,G) \) denote the set of all joint distribution functions on \( (X,Y) \) that have the given marginals. \cite{ruschendorf1983solution} proved that the set \( \mathcal{M}(F,G) \) is convex, tight, and closed. Therefore, by Prokhorov's theorem \citep{prokhorov1956convergence}, \( \mathcal{M}(F,G) \) is compact with respect to the weak topology.

\begin{proposition}\label{prop:achieve_minim_lower_semi}
    For a measurable function \( \varphi: \mathbb{R}^2 \to \mathbb{R} \) and the compact set \( \mathcal{M}(F,G) \),
    \[
    \inf_{H \in \mathcal{M}(F,G)} \int \varphi \, dH
    \]
    achieves its minimum in \( \mathcal{M}(F,G) \) when \( \varphi \) is lower semicontinuous.
\end{proposition}

\begin{proof}
Since \( \mathcal{M}(F,G) \) is compact with respect to the weak topology and \( \varphi \) is lower semicontinuous, we can apply the Portmanteau theorem. Specifically, for any sequence \( \{H_n\} \subset \mathcal{M}(F,G) \) that converges weakly to some \( H \in \mathcal{M}(F,G) \), we have
\[
\int \varphi \, dH \leq \liminf_{n \to \infty} \int \varphi \, dH_n.
\]
This means the mapping \( H \mapsto \int \varphi \, dH \) is lower semicontinuous on \( \mathcal{M}(F,G) \). Since the infimum of a lower semicontinuous function on a compact set is attained, there exists \( H^* \in \mathcal{M}(F,G) \) such that
\[
\inf_{H \in \mathcal{M}(F,G)} \int \varphi \, dH = \int \varphi \, dH^*.
\]
\end{proof}

\begin{proposition}
    For any given \( z \in \mathbb{R} \), the function \( \varphi(X,Y) = \mathbbm{1}_{\{X+Y<z\}} \) is lower semicontinuous.
\end{proposition}

\begin{proof}
Consider the function \( \varphi(x,y) = \mathbbm{1}_{\{x + y < z\}} \). 

\textit{Case 1}: If \( x + y < z \), then \( \varphi(x,y) = 1 \). For any sequence \( (x_n, y_n) \to (x,y) \), we have \( x_n + y_n < z \) for sufficiently large \( n \). Thus, \( \varphi(x_n, y_n) = 1 \) eventually, and
\[
\liminf_{(x', y') \to (x,y)} \varphi(x', y') = 1 = \varphi(x, y).
\]

\textit{Case 2}: If \( x + y \geq z \), then \( \varphi(x,y) = 0 \). For any sequence \( (x_n, y_n) \to (x,y) \), \( \varphi(x_n, y_n) \geq 0 = \varphi(x,y) \). Therefore,
\[
\liminf_{(x', y') \to (x,y)} \varphi(x', y') \geq \varphi(x, y).
\]

In both cases, the condition for lower semicontinuity is satisfied. Hence, \( \varphi \) is lower semicontinuous.
\end{proof}

Therefore, by Proposition \ref{prop:achieve_minim_lower_semi} for all \( z \in \mathbb{R} \), there exists a joint distribution function \( H_z(x,y) \) such that \( P(X + Y < z) \) under \( H_z \) equals
\[
\inf_{H \in \mathcal{M}(F,G)} P_{H}(X + Y < z),
\]
where the infimum is taken over all joint distribution functions \( H(x,y) \) with marginals \( F(x) \) and \( G(y) \). 

An analog to Proposition \ref{prop:achieve_minim_lower_semi} will give:   $\sup_{H \in \mathcal{M}(F,G)} \int \varphi \, dH$ achieves its maximum in compact set \( \mathcal{M}(F,G) \) when $\varphi$ is upper-semicontinuous. Thus, for all \( z \), the upper bound on \( P(X + Y \leq z) \) is achievable since \( \varphi = \mathbbm{1}_{\{X+Y \leq z\}} \) is upper semicontinuous.\hfill $\Box$

\begin{remark}
    Note this does not address achievability of upper bounds on $P(X+Y<z)$ or lower bounds on $P(X+Y\leq z)$ since their functions are not respectively upper and lower semi-continuous.
\end{remark}

\section{Revisit Theorem 2 in \cite{williamson1990probabilistic}}\label{App:A}
This paper was in part motivated by the observation that authors were using the lower bound given in \cite{williamson1990probabilistic}
 that were claimed to be sharp when this was clearly not true. Here we identify the error present in \cite{williamson1990probabilistic}'s proof. The proof of Theorem $2$ of \cite{williamson1990probabilistic} states ``Let $Y'=-Y$. Then $F_{Y'}(y)=1-F(-y)$". 
Since \cite{williamson1990probabilistic} use the left-continuous version definition of cdf, in our notation, 
  \begin{align*}\Tilde{G'}(y)&=P(Y'< y)\\
    &=P(-Y< y)\\
&=P(Y> -y)\\
&=1-P(Y\leq -y )\\
&=1-\Tilde{G}(-y) - \textcolor{red}{P(Y = -y)},\end{align*}
where $\Tilde{G'}(y)=P(Y'<y)$.
This statement is also not correct if we use the right-continuous version definition of cdf, as we have
    \begin{align*}G'(y)&=P(Y'\leq y)\\
    &=P(-Y\leq y)\\
&=P(Y\geq -y)\\
&=1-P(Y<-y )\\
&=1-G(-y) + \textcolor{red}{P(Y = -y)}.\end{align*}

In fact, it is easy to show that $G'(y)=1-G(-y)$ if and only if $\Tilde{G'}(y)=1-\Tilde{G}(-y)$. 

In this paper we only focus on the sum and difference of two random variables, however, a similar mistake also appears in the argument given for the bounds on the cumulative distribution function for the quotient of two random variables stated in Theorem 2 of \cite{williamson1990probabilistic}.

\section{Relating Rüschendorf and Makarov Bounds}\label{app:ruschendorf_vs_makarov}
In Section 2, Proposition 1 of \cite{ruschendorf1982random}, the upper bound \( M \) (which we denote by \(\rho_W(F,G)(z)\)) on \( A_2(t) \) (corresponding to \( P(X+Y \leq z) \) in our notation) is given by
\[
F_1 \wedge F_2(t) := \inf_x \Bigl(F_1(x-) + F_2(t-x)\Bigr),
\]
which in our notation this bound would be
\[\inf_x(F(x-)+G(z-x)) = \inf_{x+y=z}(F(x-)+G(y)).\]
In Theorem \ref{Thm:franks}, we state the upper bound as
\[
\inf_{x+y=z} \min\bigl(1, F(x)+G(y)\bigr).
\]

There are two notable differences between these formulations. First, the Makarov bound explicitly takes the minimum with 1 to ensure that the upper bound does not exceed 1, thereby avoiding a trivial bound.\footnote{The Makarov lower bound for \(P(X+Y<z)\) differs from the lower bound in \cite{ruschendorf1982random} only by taking the maximum with \(0\); we do not elaborate on this minor distinction here.}
Second, the expression in \cite{ruschendorf1982random} uses the left-hand limit \( F_1(x-) \) instead of \( F(x) \) as used in our bound \( \inf_x \bigl(F(x) + G(z-x)\bigr) \). Although it may not be immediately obvious, one can show that the two formulations are indeed equivalent.

\begin{proposition}\label{prop:eq_ru_makarov}
For any $z\in \mathbb{R}$ and distribution functions $F,G$,
    $$\inf_x(F(x)+G(z-x)) = \inf_x(F(x-)+G(z-x))$$
\end{proposition}

\begin{proof}
We will prove this equality by showing two inequalities.

Since $F$ is a distribution function, $F(x-) \leq F(x)$ for all $x$, thus
\[
F(x-) + G(z-x) \leq F(x) + G(z-x).
\]
Taking the infimum over all $x$ preserves the inequality:
\[
\inf_x (F(x-) + G(z-x)) \leq \inf_x (F(x) + G(z-x)).
\]

For the second direction, consider a point $x$ and a sequence $\{x_n\}_{n\in\mathbb{N}}$ such that $x_n \uparrow x$ (i.e., $x_n < x$ for all $n$ and $\lim_{n\to\infty} x_n = x$) and $\lim_{n\to\infty} F(x_n) = F(x-)$.
Consider the expression $F(x_n) + G(z - x_n)$. By definition of infimum, for all $n$,
    \[
    \inf_{x^*} (F(x^*) + G(z - x^*)) \leq F(x_n) + G(z - x_n).
    \]
Taking the limit as $n \to \infty$, we get
\[
 \inf_{x^*} (F(x^*) + G(z - x^*)) \leq F(x-) + \lim_{n\to\infty} G(z - x_n).
\]
As $z-x_n$ converges to $z-x$ from the right, we have $\lim_{n \to \infty} G(z-x_n) = G(z-x)$ by right continuity of $G$.
Therefore,
\[
\inf_{x^*} (F(x^*) + G(z - x^*)) \leq F(x-) + G(z-x).
\]
Taking the infimum over $x$ on the right side yields
\[
\inf_{x^*} (F(x^*) + G(z - x^*)) \leq \inf_x (F(x-) + G(z-x)).
\]
Therefore, we have:
\[
\inf_x(F(x)+G(z-x)) = \inf_x(F(x-)+G(z-x)).
\]
where the right hand side here is the upper bound in \cite{ruschendorf1982random} in our notation.
\end{proof}

In fact, there is another equivalent way to represent this upper bound as we will show next.
\begin{corollary}\label{claim:y_y_minus}
For any $z\in \mathbb{R}$ and distribution functions $F,G$,
    $$\inf_{x+y=z}(F(x)+G(y)) = \inf_{x+y=z}(F(x)+G(y-)) = \inf_y(F(z-y)+G(y-)).$$
\end{corollary}
\begin{proof}
    This follows from the symmetry of $F,G$ in the infimum expression and we can exchange $x,y$ in the sum.
\end{proof}

Note that the conclusions in Proposition~\ref{prop:eq_ru_makarov} and Corollary \ref{claim:y_y_minus} do not hold when the $\inf$ operators are replaced with $\sup$ operators. In addition, Corollary~\ref{claim:y_y_minus} provides intuition to prove Theorem~\ref{upper:equal}, as detailed in Appendix \ref{App:B}.
    

\section{Proof of Theorem \ref{upper:equal}}\label{App:B}
Here we prove, as stated in Theorem \ref{upper:equal}, that the upper bound on the cdf for the difference $\Delta=X-Y$ given by \cite{williamson1990probabilistic} is valid even though the proof in \cite{williamson1990probabilistic} is not correct. 

It is sufficient to show that for any random variables $X$ and $Y$ with respective cdfs $F(\cdot)$ and $G(\cdot)$,
\begin{align}
    \inf_{x-y=\delta} \min\{F(x)-P(Y<y),0\}
&=  \inf_{x-y=\delta} \min\{F(x)-G(y),0\}.\label{eq:inf_upperbounds_equal}
\end{align}
Recall that Proposition~\ref{claim:y_y_minus} shows    
\[
    \inf_{x+y=z}(F(x)+G(y)) = \inf_{x+y=z}(F(x)+G(y-)).
\]
Consider a new variable $Y'=-Y$ with cdf $G'$. Then for any $\delta$, applying Proposition~\ref{claim:y_y_minus} we get:
\[
    \inf_{x+y'=\delta}(F(x)+G'(y')) = \inf_{x+y'=\delta}(F(x)+G'(y'-)).
\]
For $y'=-y$, we see that 
\begin{align*}G'(y'-)&=P(Y'<y')\\
&=P(-Y< y')\\
&=P(Y> -y')\\
&=1-P(Y\leq -y' )\\
&=1-G(-y') \\
&= 1-G(y).\end{align*}
And similarly $G'(y')=P(-Y\leq y')=P(Y\geq -y')=1-P(Y< -y' ) = 1-P(Y<y)$, so 
\[
    \inf_{x-y=\delta}(F(x)-P(Y<y)+1) = \inf_{x-y=\delta}(F(x)+G(y)+1).
\]
Because adding or subtracting a constant and taking minimum do not affect the relevant infimum.
Thus, the two expressions in (\ref{eq:inf_upperbounds_equal}) coincide.

\section{Simulation under the special copula}\label{app:sim_copula}
We note that the special copula $C_t$ which achieves the lower bound on $P(X+Y<z)$ for $t=\sup_{x+y=z} \max({F(x)+G(y)-1},0)$ in \cite{frank1987best,nelsen2006introduction} is degenerate in the sense that it does not have a density with respect to the two‐dimensional Lebesgue measure over the unit square. Recall that it is defined as:
$$
C_t(u, v)= \begin{cases}\operatorname{Max}(u+v-1, t), & (u, v) \text { in }[t, 1] \times[t, 1], \\ \operatorname{Min}(u, v), & \text { otherwise.}\end{cases}
$$
Though possibly not immediately obvious, under this copula, $V = G(Y)$ is almost surely a deterministic (piecewise-defined) function of $U=F(X)$ (and vice versa).\footnote{This does not mean $X$ is a deterministic function of $Y$ unless they are both absolutely continuous with respect to the Lebesgue measure.} Thus, we can easily simulate the joint distribution from the copula with a single draw. We introduce the following algorithm:
\begin{algorithm}[H]
\caption{Simulate Joint Distribution from the Copula $C_t$}\label{alg:copula}
\begin{algorithmic}
\State Compute $t=\sup_{x+y=z} \max({F(x)+G(y)-1},0)$
\For{$i = 1, \dots, n$}
    \State Draw $u_i \sim \mathrm{Unif}[0,1]$.
    \State Compute 
    \[
    v_i = u_i\mathbbm{1}\{u_i \leq t\} + (1+t-u_i) \mathbbm{1}\{u_i > t\}.
    \]
    \State Set $x_i = F^{-1}(u_i)$ and $y_i = G^{-1}(v_i)$, where the generalized inverses are defined in Definition~\ref{def:g_inv}.
    \State Take $(x_i,y_i)$ as a generated sample.
\EndFor
\end{algorithmic}
\end{algorithm}

Similarly, the joint distribution under copula $C_r$, which achieves the upper bound on $P(X+Y\leq z)$ with $r=\inf_{x+y=z} \min\bigl(1, F(x)+G(y)\bigr)$ is defined as:
$$
C_r(u, v)= \begin{cases}\operatorname{Max}(u+v-r, 0), & (u, v) \text { in }[0, r] \times[0, r], \\ \operatorname{Min}(u, v), & \text { otherwise.}\end{cases}
$$
$C_r$ can be simulated similarly using the deterministic function $v_i = (r-u_i)\mathbbm{1}\{u_i \leq r\} +  u_i \mathbbm{1}\{u_i > r\} $. Here, copulas $C_t$ and $C_r$ are ordinal sums \citep{mesiar2010ordinal, nelsen2006introduction} of the singular copulas $W(u,v)=\max(u+v-1,0)$ and $M(u,v)=\min(u,v)$ corresponding to the Fréchet-Hoeffding lower and upper bounds. 

As noted in \cite{frank1987best}, copulas that achieve the bounds in certain scenarios are not unique. For example, $\operatorname{Min}(u, v)$ in $[0,t]\times[0,t]$ of $C_t$ can be replaced by $uv$. Similarly, $\operatorname{Min}(u, v)$ in $[r,1]\times[r,1]$ of $C_r$ can also be replaced by $uv$.
However, each of these copulas is an ordinal sum of two copulas, with at least one being degenerate. As a result, the induced joint distribution may fail to be continuous—even when the marginal distributions are continuous. Consequently, the Makarov bounds need not always be attainable under these copulas (see Theorem~\ref{thm:sum_mak}). In contrast, when at least one of the marginals is discrete, the Makarov bounds are always attained (see Theorem~\ref{thm:discrete_ach}).
\end{appendix}

\end{document}